\documentclass[11pt]{amsart}
\usepackage[margin=1.00in]{geometry}
\usepackage{amsmath,amssymb,amsthm,mathtools}
\usepackage[hidelinks]{hyperref}
\usepackage{cleveref}
\usepackage{aliascnt}
\usepackage{mathrsfs}
\usepackage{esint}
\usepackage{booktabs}

\numberwithin{equation}{section}

\usepackage{enumerate}

\usepackage[utf8]{inputenc}

\usepackage{mathtools}

\usepackage{fancyhdr}
\usepackage{hyperref}
\usepackage{url}

\newtheorem{theorem}{Theorem}[section]

\newtheorem{lemma}[theorem]{Lemma}
\newtheorem{corollary}[theorem]{Corollary}
\newtheorem{remark}[theorem]{Remark}

\newcommand{\R}{\mathbb R}
\newcommand{\1}{\mathbf 1}

\newcommand{\abs}[1]{\lvert #1\rvert}
\newcommand{\norm}[1]{\lVert #1\rVert}

\newcommand{\diam}{\operatorname{diam}}

\newcommand{\Gt}{\mathcal{G}}

\def\Pt{\mathcal{P}}

\def\EE{\mathbf{E}}

\newcommand{\C}{\mathbb{C}}

\newcommand{\beq}{\begin{equation}}
	\newcommand{\enq}{\end{equation}}

\newcommand{\eq}[1]{\begin{align*}#1\end{align*}}

\def\be{{\beta}}

\def\og{{\omega}}

\def\Bl{\Bigl}
\def\Br{\Bigr}
\def\f{\frac}

\def\sb{{\mathbf s}}
\def\tb{{\mathbf t}}

\def\CD{{\mathcal D}}

\def\cP{{\mathcal P}}

\def\PP{{\mathbf P}}

\def\ZZ{{\mathbb Z}}

\def\Og{\Omega}

\newcommand{\wt}{\widetilde}

\def\cA{\mathcal{A}}

\def\bx{\boldsymbol{x}}
\def\by{\boldsymbol{y}}
\def\bz{\boldsymbol{z}}

\def\bk{\boldsymbol{k}}
\def\bfa{\boldsymbol{a}}

\def\bw{\boldsymbol{w}}
\def\bn{\boldsymbol{n}}
\def\bm{\boldsymbol{m}}
\def\bL{\boldsymbol{L}}

\newcommand{\bp}{ \begin{proof} }
	
	\newcommand{\ep}{  \end{proof} }

\def\be{\beta}

\def\dfrac{\displaystyle\frac}

\def\og{\omega}

\def\Bl{\Bigl}
\def\Br{\Bigr}
\def\f{\frac}

\def\FF{\mathcal{F}}

\def\mathscr{\mathcal}

\def\XXint#1#2#3{{\setbox0=\hbox{$#1{#2#3}{\int}$ }
		\vcenter{\hbox{$#2#3$ }}\kern-.55\wd0}}

\newcommand{\T}{\mathbb T}
\newcommand{\Z}{\mathbb Z}
\newcommand{\N}{\mathbb N}
\newcommand{\cT}{\mathcal T}

\newcommand{\ii}{\mathrm i}

\newcommand{\Prb}{\mathbf{P}}

\newcommand{\Qt}{\mathcal{Q}}

\begin{document}

\title[Uniform norm discretization for hyperbolic crosses]{Sampling discretization of the uniform norm for\\ 
hyperbolic-cross trigonometric polynomials}
\author{Feng Dai}
\address{F.~Dai, Department of Mathematical and Statistical Sciences\\
	University of Alberta\\ Edmonton, Alberta T6G 2G1, Canada.}
\email{fdai@ualberta.ca}

\author{Andriy Prymak}
\address{Department of Mathematics, University of Manitoba, Winnipeg, MB, R3T2N2, Canada}

\email{prymak@gmail.com}

\begin{abstract}
We study sampling discretization of the uniform norm for trigonometric
polynomials with frequencies in a hyperbolic cross of level $N$ on $\T^d$.
For every $d,N\ge2$ and $\varepsilon\in(0,1)$, we construct an explicit
norming set with at most
$\bigl(C_1(1+\varepsilon^{-1})\bigr)^{d-1}N^{1+\varepsilon}$ points and
norming constant at most
$\bigl(C_2(1+\varepsilon^{-1})\bigr)^{d-1}$, where $C_1,C_2$ are absolute
constants. We also prove that independent Haar-distributed points achieve
the same exponent $1+\varepsilon$: a sample of size at least
$C(d,\varepsilon)N^{1+\varepsilon}\log(2/\eta)$ norms the entire space
simultaneously with probability at least $1-\eta$, with a norming constant
bounded by $\exp(C_d\varepsilon^{-d})$ and independent of $N$.
The deterministic construction combines uniformly stable de la Vall\'ee
Poussin sampling operators with a Smolyak-type combination identity.
The random result follows from an abstract norming theorem for sums of
spaces associated with commuting regular partitions, together with a
multiscale approximation of hyperbolic-cross polynomials.
Finally, we show that every norming set of constant $B>1$ for the
trigonometric polynomials with frequencies in $\{-1,0,1\}^d$ has at least
$\bigl(\pi d/(2e\log B)\bigr)^{d/2}$ points. This rules out quadratic
bounds in the dimension of the polynomial space with a prefactor growing
only polynomially in $d$ when the norming constant is fixed.
\end{abstract}

\maketitle

\tableofcontents

\section{Introduction}

Let $\T^d$ denote the $d$-dimensional torus, identified with $[0,2\pi)^d$ and equipped with the normalised Haar measure
\[
 d\mu_d(\bx)=(2\pi)^{-d}\,d\bx.
\]
For $f\in L^1(\T^d)$, its Fourier coefficients are defined by
\[
 \widehat f(\bk):=\int_{\T^d}f(\bx)e^{-\ii\langle \bk,\bx\rangle}\,d\mu_d(\bx),
 \qquad \bk\in\Z^d.
\]
Given a finite set $Q\subset\Z^d$, let
\[
 \cT(Q):=\left\{\sum_{\bk\in Q}c_{\bk}e^{\ii\langle \bk,\cdot\rangle}:
 c_{\bk}\in\C\right\}.
\]
We write $\cT$ for the space of all trigonometric polynomials on $\T^d$. For an integer $N\ge1$, define the hyperbolic cross
\[
 \Gamma(N,d):=\left\{\bk=(k_1,\ldots,k_d)\in\Z^d:
 \prod_{j=1}^d\max\{|k_j|,1\}\le N\right\}
\]
and set
\[
 \cT(N,d):=\cT\bigl(\Gamma(N,d)\bigr).
\]
For every fixed $d\ge2$,
\begin{equation}\label{eq:dim-HC-intro}
 \dim\cT(N,d)=|\Gamma(N,d)|
 =\frac{2^d}{(d-1)!}N(\log N)^{d-1}
 +O_d\bigl(N(\log N)^{d-2}\bigr),
 \qquad N\to\infty.
\end{equation}
In particular,
\[
 \dim\cT(N,d)\asymp_d N\log^{d-1}(2N).
\]

For $f\in C(\T^d)$ and a nonempty compact set $E\subset\T^d$, write
\[
 \|f\|_E:=\max_{\bx\in E}|f(\bx)|.
\]
Let $V$ be a linear subspace of $C(\T^d)$. We say that $V$ admits an $m$-point sampling discretization of the uniform norm with constant $B\ge1$ if there exists a finite set $W\subset\T^d$, with $|W|\le m$, such that
\begin{equation}\label{eq:norming-intro}
 \|f\|_\infty\le B\|f\|_W,
 \qquad f\in V.
\end{equation}
Such a set $W$ will be called a norming set of constant $B$ for $V$. The basic problem is to estimate the smallest possible cardinality of a norming set while keeping the norming constant under control. We refer to \cite{DPTT,KKT,KT2024} and the references therein for background on sampling discretization of the uniform norm.

The case of hyperbolic-cross trigonometric polynomials is of particular interest. According to \cite[Theorem~5.1]{DPTT}, there exist positive constants $B_d$ and $C_d$, depending only on $d$, such that $\cT(N,d)$ admits a norming set of constant $B_d$ and cardinality
\begin{equation}\label{eq:old}
 m\le C_dN^{H_d}(\log N)^{d-H_d},
 \qquad
 H_d:=\sum_{j=1}^d\frac1j.
\end{equation}
On the other hand, when $d=2$, the lower bounds obtained in \cite{KT1998,KT1999,KT2003} imply that if $\cT(N,2)$ admits an $m$-point sampling discretization with constant $B\ge1$, then (see also \cite[Theorem~5.2]{DPTT})
\begin{equation}\label{1-3-eq}
 m\ge c_1N^{1+c_2B^{-2}}\log N,
\end{equation}
where $c_1,c_2>0$ are absolute constants. Thus, even in dimension two, a fixed norming constant does not permit a bound of order $N$ up to logarithmic factors.

The first main result of this paper gives an explicit norming set whose cardinality is arbitrarily close to $N$ in the power scale. It improves \eqref{eq:old} for every fixed $d\ge2$ and every sufficiently small $\varepsilon>0$.

\begin{theorem}\label{thm-1-1}
There exist absolute constants $C_1,C_2\ge1$ with the following property. For every $\varepsilon\in(0,1)$ and all integers $d,N\ge2$, there exists an explicitly constructed finite set $W_{d,\varepsilon}(N)\subset\T^d$ such that
\begin{equation}\label{eq:generalcard-1-4}
 |W_{d,\varepsilon}(N)|
 \le \bigl(C_1(1+\varepsilon^{-1})\bigr)^{d-1}N^{1+\varepsilon}
\end{equation}
and
\begin{equation}\label{eq:generalnorm-1-5}
 \|f\|_\infty
 \le \bigl(C_2(1+\varepsilon^{-1})\bigr)^{d-1}
 \|f\|_{W_{d,\varepsilon}(N)},
 \qquad f\in\cT(N,d).
\end{equation}
\end{theorem}

Several comments are in order. First, for fixed $d$ and $\varepsilon$, the constants in Theorem~\ref{thm-1-1} are independent of $N$. Since $H_d>1$ for $d\ge2$, the estimate \eqref{eq:generalcard-1-4} improves the power of $N$ in \eqref{eq:old} whenever $0<\varepsilon<H_d-1$. The price for taking $\varepsilon$ small is the deterioration of the norming constant and of the multiplicative constant in the cardinality estimate.

Second, Theorem~\ref{thm-1-1} is compatible with the two-dimensional lower bound \eqref{1-3-eq}. Indeed, for $d=2$, the norming constant in \eqref{eq:generalnorm-1-5} is
\[
 B=C_2(1+\varepsilon^{-1}).
\]
Substitution into \eqref{1-3-eq} gives
\[
 |W_{2,\varepsilon}(N)|
 \ge c_1N^{1+c_3(\varepsilon)}\log N,
\qquad
 c_3(\varepsilon)
 :=c_2\left(\frac{\varepsilon}{C_2(1+\varepsilon)}\right)^2.
\]
Since
\[
 c_3(\varepsilon)\le \frac{c_2}{C_2}\varepsilon,
\]
after increasing $C_2$ if necessary, we may assume $C_2>c_2$ and hence $c_3(\varepsilon)<\varepsilon$. Therefore, the lower bound above does not contradict \eqref{eq:generalcard-1-4} as $N\to\infty$.

Our second result concerns the dependence on the dimension. The constants in \eqref{eq:generalcard-1-4} and \eqref{eq:generalnorm-1-5} grow exponentially with $d$. It is therefore natural to ask whether, for a fixed $B>1$, one can find a norming set $W$ satisfying
\[
 |W|\le C(d)\bigl(\dim\cT(N,d)\bigr)^2
\]
with $C(d)$ growing at most polynomially in $d$. We show that this is impossible already for $N=1$. Since
\[
 \Gamma(1,d)=\{-1,0,1\}^d,
 \qquad
 \dim\cT(1,d)=3^d,
\]
Theorem~\ref{thm:HC-obstruction} gives the following lower bound: every finite set $W\subset\T^d$ satisfying
\[
 \|f\|_\infty\le B\|f\|_W,
 \qquad f\in\cT(1,d),
\]
must satisfy
\begin{equation}\label{eq:HC-lower-intro}
 |W|\ge
 \left(\frac{\pi d}{2e\log B}\right)^{d/2}
 =\left(\frac{\pi d}{162e\log B}\right)^{d/2}
 \bigl(\dim\cT(1,d)\bigr)^2.
\end{equation}
Consequently, if
\[
 |W|\le C(d)\bigl(\dim\cT(1,d)\bigr)^2,
\]
then necessarily
\begin{equation}\label{eq:B-growth-intro}
 \log B\ge \frac{\pi d}{162e\,C(d)^{2/d}}.
\end{equation}
In particular, if $C(d)$ is bounded by a polynomial in $d$, then $B\ge e^{cd}$ for some absolute constant $c>0$ and all sufficiently large $d$.

We next turn to random sampling. Let $\xi^1,\xi^2,\ldots$ be independent random points distributed according to $\mu_d$, and put
\[
 X_m:=\{\xi^1,\ldots,\xi^m\}.
\]
A general transfer principle proved in Section~5 converts a deterministic norming set for a translation-invariant subspace into a random norming set. Combined with Theorem~\ref{thm-1-1}, it yields the sufficient bound $m\ge C(d,\varepsilon)N^{2+\varepsilon}(1+\log(1/\eta))$. Our third main result shows that the extra factor of order $N$ can be removed. Thus, random points achieve the same exponent $1+\varepsilon$ as the explicit deterministic construction.

\begin{theorem}\label{thm:main}
Let $d,N\ge2$, $\varepsilon\in(0,1)$, and $\eta\in(0,1)$. There exist constants $C(d,\varepsilon)>0$ and $B_{d,\varepsilon}>0$ with
$
 B_{d,\varepsilon}\le
 \exp\bigl(C_d\varepsilon^{-d}\bigr)$ such that,  whenever
\begin{equation}\label{eq:m-random-intro}
 m\ge C(d,\varepsilon)\log\frac2\eta\,N^{1+\varepsilon},
\end{equation}
the inequality
\begin{equation}\label{eq:random-norming-intro}
 \|f\|_\infty\le
 B_{d,\varepsilon} \|f\|_{X_m}
\end{equation}
holds simultaneously for all $f\in\cT(N,d)$ with probability at least $1-\eta$.
\end{theorem}

\begin{remark}
The deterministic and random results have the same exponent $1+\varepsilon$ in $N$, but their norming constants have very different dependence on \(\varepsilon\). 
\end{remark}

We briefly describe the main ideas of the proofs. The proof of Theorem~\ref{thm-1-1} is based on uniformly stable sampling operators associated with the de la Vall\'ee Poussin kernels. We choose a geometric sequence of frequency scales and apply a Smolyak-type combination identity to the resulting tensor-product operators. The hyperbolic-cross condition forces all active multi-indices to lie in a simplex of bounded width. The norming set is then obtained as a union of tensor-product grids corresponding to the boundary layers of this simplex. This replaces the use of the mixed Bernstein inequality in the proof of \eqref{eq:old} and leads to the estimate $N^{1+\varepsilon}$.

The proof of the dimension-dependent obstruction \eqref{eq:HC-lower-intro} uses the positive trigonometric polynomial
\[
 F(\bx)=\prod_{j=1}^d\cos^2\left(\frac{x_j}{2}\right).
\]
Translation invariance implies that translates of the superlevel set $\{F\ge B^{-1}\}$ centred at the points of a norming set must cover the torus. A Gaussian upper bound for $F$ then reduces the problem to a volume estimate for a Euclidean ball.

The proof of Theorem~\ref{thm:main} has two main components. First, we establish an abstract random norming theorem for a sum
\[
 V=H_1+\cdots+H_K
\]
of spaces of functions that are constant on regular partitions of a finite probability space, under the assumption that the corresponding conditional expectation operators commute. Matrix Chernoff bounds provide an empirical $L^2$ frame on $V$. The empirical conditional expectations are contractions in the uniform norm, and an annihilating polynomial for the average of the commuting projections allows us to pass from the empirical $L^2$ control to uniform norming.

Second, we transfer this abstract result to $\cT(N,d)$. A  de la Vall\'ee Poussin decomposition writes a hyperbolic-cross polynomial as a sum of dyadic frequency blocks. Each block is approximated uniformly by a function that is constant on a suitable dyadic partition. The active frequency indices are grouped into a bounded number, depending only on $d$ and $\varepsilon$, of macro-families whose associated conditional expectations commute. The resulting piecewise-constant model is then normed by random points through the abstract theorem, and the approximation error is absorbed into the norming inequality.

The paper is organised as follows. In Section~2 we introduce the uniformly stable sampling operators used in the deterministic construction. Section~3 contains the proof of Theorem~\ref{thm-1-1}. In Section~4 we prove the high-dimensional obstruction stated in Theorem~\ref{thm:HC-obstruction}. Section~5 gives the transfer principle from deterministic to random norming. Section~6 establishes the abstract random norming theorem for commuting regular partitions. In Section~7 we construct the dyadic grids, the de la Vall\'ee Poussin decomposition, and in Section 8.1, we construct  the macro-partitions needed for the transfer to hyperbolic crosses.   Theorem~\ref{thm:main} is proved in Section~8.

\section{Uniformly stable sampling operators}

The proof of \eqref{eq:old} given in \cite{DPTT} is based on the mixed Bernstein inequality
\[
 \norm{\partial_1\cdots\partial_d f}_\infty
 \le C(d)N(\log N)^{d-1}\norm{f}_\infty,
 \qquad f\in \cT(N,d).
\]
The factor $(\log N)^{d-H_d}$ in \eqref{eq:old} arises from this derivative estimate. In the present argument, we replace this step by a family of exact sampling operators whose norms are bounded uniformly in the degree.

We begin with the one-dimensional case. For an integer $n\ge1$, let
\[
 F_n(t):=\frac1{n+1}
 \left(\frac{\sin((n+1)t/2)}{\sin(t/2)}\right)^2
\]
be the Fej\'er kernel of order $n$, and define the de la Vall\'ee Poussin kernel
\beq
 K_n(t):=2F_{2n-1}(t)-F_{n-1}(t).
\enq
We also set $K_0(t)=1$. For $n\ge1$, define the sampling operator $V_n:C(\T)\to C(\T)$ by
\begin{equation}\label{eq:SM}
 V_nf(x):=\frac1{4n}\sum_{j=0}^{4n-1}
 f\left(\frac{\pi j}{2n}\right)
 K_n\left(x-\frac{\pi j}{2n}\right).
\end{equation}
For completeness, set $V_0f:=f(0)$. We shall use the grid
\begin{equation}\label{eq:grid-Gn}
 G_n:=\left\{\frac{j\pi}{2n}:j=0,1,\ldots,4n-1\right\},
 \qquad n\ge1.
\end{equation}
The following standard properties of $V_n$ will be used repeatedly:
\begin{enumerate}[\rm (i)]
 \item $V_nf=f$ for every $f\in\cT([-n,n])$;
 \item $V_nf\in\cT([-2n+1,2n-1])$ for every $f\in C(\T)$;
 \item there exists an absolute constant $A\ge1$ such that
 \begin{equation}\label{eq:Vn-stability}
  \norm{V_nf}_\infty\le A\max_{\omega\in G_n}|f(\omega)|,
  \qquad f\in C(\T),\quad n\ge1;
 \end{equation}
 \item $\lim_{n\to\infty}\|V_nf- f\|_{L^\infty(\T)}=0$ for every $f\in C(\T)$. Equivalently,
 \[
  f=V_0f+\sum_{n=1}^\infty(V_nf-V_{n-1}f),
 \]
 where the series converges uniformly.
\end{enumerate}

We next pass to several variables. Given an operator $U:C(\T)\to C(\T)$ and an integer $1\le j\le d$, let $U^{(j)}:C(\T^d)\to C(\T^d)$ denote the operator obtained by applying $U$ in the $j$th variable. Namely,  for $f\in C(\T^d)$ and $\bx=(x_1,\ldots,x_d)\in\T^d$,
\[
 U^{(j)}f(\bx)
 :=U\bigl[f(x_1,\ldots,x_{j-1},\,\cdot\,,x_{j+1},\ldots,x_d)\bigr](x_j).
\]
Clearly, the operators $U^{(j)}$ acting in distinct variables commute. In particular, property (i) above gives
\begin{equation}\label{eq:tensor-reproduction}
 V_{n_1}^{(1)}\cdots V_{n_d}^{(d)}f=f,
 \qquad
 f\in\cT\bigl([-n_1,n_1]\times\cdots\times[-n_d,n_d]\bigr).
\end{equation}

\section{Proof of Theorem~\ref{thm-1-1}}

Fix $d\ge2$. For a multi-index $\nu=(\nu_1,\ldots,\nu_d)\in\Z_+^d$, write
\[
 |\nu|_1:=\nu_1+\cdots+\nu_d.
\]
We first record the Smolyak combination identity in the form needed below.

\begin{lemma}[Smolyak combination identity]\label{lem:combination}
Let $P_j:C(\T)\to C(\T)$, $j=0,1,\ldots$, be a sequence of linear operators, and set
\[
 D_0:=P_0,
 \qquad
 D_j:=P_j-P_{j-1},\quad j\ge1.
\]
Then, for every integer $n\ge d-1$,
\begin{equation}\label{eq:abstractcombination}
 \sum_{|\nu|_1\le n}
 D_{\nu_1}^{(1)}\cdots D_{\nu_d}^{(d)}
 =\sum_{j=0}^{d-1}(-1)^j\binom{d-1}{j}
 \sum_{|\nu|_1=n-j}
 P_{\nu_1}^{(1)}\cdots P_{\nu_d}^{(d)},
\end{equation}
where all sums are taken over $\nu\in\Z_+^d$.
\end{lemma}

For completeness, we include a proof of Lemma \ref{lem:combination} below.
\begin{proof}
Since only the operators $P_0,\ldots,P_n$ occur in \eqref{eq:abstractcombination}, we may assume that $P_k=0$ for $k>n$. For $1\le i\le d$, define the operator-valued polynomials
\[
 \mathcal P^{(i)}(z):=\sum_{m\ge0}P_m^{(i)}z^m,
 \qquad
 \mathcal D^{(i)}(z):=\sum_{m\ge0}D_m^{(i)}z^m.
\]
By the definition of $D_m$,
\[
 \mathcal D^{(i)}(z)=(1-z)\mathcal P^{(i)}(z).
\]
Consequently,
\begin{equation}\label{3-2}
 \frac1{1-z}\prod_{i=1}^d\mathcal D^{(i)}(z)
 =(1-z)^{d-1}\prod_{i=1}^d\mathcal P^{(i)}(z).
\end{equation}
The coefficient of $z^n$ on the left-hand side of \eqref{3-2} is
\begin{equation}\label{3-3}
 \sum_{|\nu|_1\le n}D_{\nu_1}^{(1)}\cdots D_{\nu_d}^{(d)}.
\end{equation}
On the other hand, since
\[
 (1-z)^{d-1}=\sum_{j=0}^{d-1}(-1)^j\binom{d-1}{j}z^j,
\]
the coefficient of $z^n$ on the right-hand side is
\begin{equation}\label{3-4}
 \sum_{j=0}^{d-1}(-1)^j\binom{d-1}{j}
 \sum_{|\nu|_1=n-j}P_{\nu_1}^{(1)}\cdots P_{\nu_d}^{(d)},
\end{equation}
where an inner sum is understood to be zero when $n-j<0$. Comparing \eqref{3-3} and \eqref{3-4} proves the lemma.
\end{proof}

\begin{proof}[Proof of Theorem~\ref{thm-1-1}]
Fix $\varepsilon\in(0,1)$ and set
\[
 r:=\left\lceil\frac{d-1}{\varepsilon}\right\rceil,
 \qquad
 q:=N^{1/r},
 \qquad
 M_j:=\lceil q^j\rceil,
 \quad 0\le j\le r.
\]
Then $M_0=1$ and $M_r=N$. Define
\[
 T_j:=V_{M_j},\quad 0\le j\le r,
 \qquad
 T_j:=T_r=V_N,\quad j>r,
\]
and put
\[
 \Delta_0:=T_0,
 \qquad
 \Delta_j:=T_j-T_{j-1},\quad j\ge1.
\]
Since $T_r$ reproduces every trigonometric polynomial of degree at most $N$, we have
\begin{equation}\label{decom}
 g=T_rg=\sum_{j=0}^r\Delta_jg
   =\sum_{j=0}^\infty\Delta_jg,
 \qquad g\in\cT([-N,N]).
\end{equation}
Moreover, $\Delta_j=0$ for $j>r$, and the reproduction property of $V_n$ gives
\begin{equation}\label{3-5:eq}
 \Delta_j e^{\ii kt}=0
 \qquad\text{whenever }j\ge 1\text{ and }|k|\le q^{j-1}.
\end{equation}

For $\nu=(\nu_1,\ldots,\nu_d)\in\Z_+^d$, define
\[
 T_\nu:=T_{\nu_1}^{(1)}\cdots T_{\nu_d}^{(d)},
 \qquad
 \Delta_\nu:=\Delta_{\nu_1}^{(1)}\cdots\Delta_{\nu_d}^{(d)}.
\]
Applying \eqref{decom} in each variable yields
\begin{equation}\label{3-7-0}
 f=\sum_{\nu_1=0}^r\cdots\sum_{\nu_d=0}^r\Delta_\nu f
  =\sum_{\nu\in\Z_+^d}\Delta_\nu f,
 \qquad f\in\cT(N,d).
\end{equation}

Let
\[
 L:=r+d-1.
\]
We claim that
\begin{equation}\label{3-7}
 \Delta_\nu f=0\   \ \text{whenever}
 \qquad f\in\cT(N,d)\  \ \text{and}\  \  |\nu|_1>L.
\end{equation}
 Assume that $\Delta_\nu f\neq 0$ for some
$f\in\cT(N,d)$. Then there must exist a vector   $\bk=(k_1,\ldots,k_d)\in\Gamma(N,d)$
such that
\[
 \Delta_\nu e^{\ii\langle\bk,\bx\rangle}=\prod_{j=1}^d  \Delta_{\nu_j}[ e^{\ii k_j x_j }] \neq 0.
\]
By \eqref{3-5:eq}, this implies that
$|k_j|> q^{v_j-1}$  whenever  $\nu_j\ge 1$, whereas if $\nu_j=0$, then   
$$\max\{|k_j|, 1\}\ge 1 > q^{-1}=q^{v_j-1}.$$  Consequently, in either case, 
\[ \max\{|k_j|,1\}> q^{v_j-1},\   \  j=1,\ldots, d.\]
Multiplying these inequalities over $j=1,\ldots, d$ and using $\bk\in\Gamma(N,d)$, we obtain
\[
 q^{v_1+\cdots+v_d-d}
 <\prod_{i=1}^d\max\{|k_i|,1\}
 \le N=q^r.
\]
This implies
\[
 v_1+\cdots+v_d\le r+d-1=L,
\]
and hence  prove  the claim \eqref{3-7}.

Combining \eqref{3-7-0} and \eqref{3-7}, we obtain
\[
 f=\sum_{|\nu|_1\le L}\Delta_\nu f.
\]
Lemma~\ref{lem:combination}, applied with $P_j=T_j$ and $D_j=\Delta_j$, now gives the exact representation
\begin{equation}\label{eq:generalcombination}
 f=\sum_{j=0}^{d-1}(-1)^j\binom{d-1}{j}
 \sum_{|\nu|_1=L-j}T_\nu f.
\end{equation}

For $0\le i\le r$, set $\Lambda_i:=G_{M_i}$, and set $\Lambda_i:=G_N$ for $i>r$. Given $\nu\in\Z_+^d$, define the tensor-product grid
\[
 W_\nu:=\Lambda_{\nu_1}\times\cdots\times\Lambda_{\nu_d}.
\]
Iterating \eqref{eq:Vn-stability} in the $d$ variables gives
\begin{equation}\label{3-13}
 \norm{T_\nu\varphi}_\infty
 \le A^d\norm{\varphi}_{W_\nu},
 \qquad \varphi\in C(\T^d).
\end{equation}
Indeed, one first applies \eqref{eq:Vn-stability} in the first variable, then in the second variable, and so on; after $d$ steps, the right-hand side is the maximum of $|\varphi|$ over $W_\nu$.

Define
\begin{equation}\label{eq:generalW}
 W_{d,r}(N):=
 \bigcup_{j=0}^{d-1}\
 \bigcup_{|\nu|_1=L-j}W_\nu.
\end{equation}
Using \eqref{eq:generalcombination}, \eqref{3-13}, and the identity
\[
 \bigl|\{\nu\in\Z_+^d:|\nu|_1=s\}\bigr|
 =\binom{s+d-1}{d-1},
\]
we obtain
\begin{align*}
 \norm f_\infty
 &\le A^d\norm f_{W_{d,r}(N)}
 \sum_{j=0}^{d-1}\binom{d-1}{j}
 \binom{L-j+d-1}{d-1}\\
 &\le 2^{d-1}A^d\binom{r+2d-2}{d-1}
 \norm f_{W_{d,r}(N)}.
\end{align*}
The standard estimate
\[
 \binom ab\le\left(\frac{ea}{b}\right)^b,
 \qquad a>b\ge1,
\]
together with $r\le(d-1)\varepsilon^{-1}+1$, gives
\begin{align*}
 2^{d-1}A^d\binom{r+2d-2}{d-1}
 &\le 2^{d-1}A^d
 \left(\frac{e(r+2d-2)}{d-1}\right)^{d-1}\\
 &\le \bigl(6eA^2(1+\varepsilon^{-1})\bigr)^{d-1}.
\end{align*}
Therefore,
\begin{equation}\label{eq:generalnorm-proof}
 \norm f_\infty
 \le \bigl(6eA^2(1+\varepsilon^{-1})\bigr)^{d-1}
 \norm f_{W_{d,r}(N)},
 \qquad f\in\cT(N,d).
\end{equation}

It remains to estimate the cardinality of $W_{d,r}(N)$. Since $M_j\le2q^j$ for $0\le j\le r$, for every $\nu\in\Z_+^d$ we have
\[
 |W_\nu|
 =\prod_{i=1}^d|G_{M_{\min\{\nu_i,r\}}}|
 \le 8^d q^{|\nu|_1}.
\]
For the multi-indices occurring in \eqref{eq:generalW}, one has $r\le|\nu|_1\le r+d-1$. Hence
\begin{align*}
 |W_{d,r}(N)|
 &\le\sum_{s=r}^{r+d-1}\sum_{|\nu|_1=s}|W_\nu|\\
 &\le 8^d q^{r+d-1}
 \sum_{s=r}^{r+d-1}\binom{s+d-1}{d-1}\\
 &\le 8^d d\,N^{1+(d-1)/r}
 \binom{r+2d-2}{d-1}\\
 &\le \bigl(96(1+\varepsilon^{-1})\bigr)^{d-1}N^{1+\varepsilon}.
\end{align*}
Finally, set
\[
 W_{d,\varepsilon}(N):=W_{d,r}(N).
\]
Then \eqref{eq:generalnorm-proof} and the last cardinality estimate prove \eqref{eq:generalnorm-1-5} and \eqref{eq:generalcard-1-4}, respectively, with $C_1=96$ and $C_2=6eA^2$.
\end{proof}

\section{Dependence on the dimension}

The constants in \eqref{eq:generalcard-1-4} and \eqref{eq:generalnorm-1-5} grow exponentially with $d$. This raises the following question. Given a fixed norming constant $B>1$, can one find a norming set $W$ for $\cT(N,d)$ such that
\[
 |W|\le C(d)\bigl(\dim\cT(N,d)\bigr)^2,
\]
where $C(d)$ grows at most polynomially in $d$? We show that this is impossible already for $N=1$. Recall that
\[
 \Gamma(1,d)=\{-1,0,1\}^d,
 \qquad
 \dim\cT(1,d)=3^d.
\]

\begin{theorem}\label{thm:HC-obstruction}
Let $B>1$, and let $W\subset\T^d$ be a finite set satisfying
\begin{equation}\label{4-1}
 \norm f_\infty\le B\norm f_W,
 \qquad f\in\cT(1,d).
\end{equation}
Then
\begin{equation}\label{eq:HC-lower}
 |W|\ge
 \left(\frac{\pi d}{2e\log B}\right)^{d/2}
 =\left(\frac{\pi d}{162e\log B}\right)^{d/2}
 \bigl(\dim\cT(1,d)\bigr)^2.
\end{equation}
\end{theorem}

\begin{proof}
For $\bx=(x_1,\ldots,x_d)\in[-\pi,\pi)^d$, define
\[
 F(\bx):=\prod_{j=1}^d\cos^2\left(\frac{x_j}{2}\right)
 =\frac1{4^d}\prod_{j=1}^d\bigl(e^{\ii x_j}+e^{-\ii x_j}+2\bigr).
\]
Then $F\in\cT(1,d)$, $0\le F\le1$, and $\norm F_\infty=F(0)=1$. Since $\cT(1,d)$ is translation invariant, $F(\cdot-\bfa)\in\cT(1,d)$ for every $\bfa\in\T^d$. It follows from \eqref{4-1} that, for every $\bfa\in\T^d$, there exists $\bw\in W$ such that
\[
 F(\bw-\bfa)\ge B^{-1}.
\]
Set
\[
 E_B:=\{\bx\in[-\pi,\pi)^d:F(\bx)\ge B^{-1}\}.
\]
The preceding observation gives the covering
\[
 [-\pi,\pi)^d=\bigcup_{\bw\in W}(\bw-E_B),
\]
and hence
\begin{equation}\label{4-3}
 |W|\mu_d(E_B)\ge1.
\end{equation}

It remains to estimate the measure of $E_B$. Since for $|t|\le\pi$,
\[
 \cos (t/2)\le e^{-t^2/8},
\]
it follows that
\[
 F(\bx)\le\exp\left(-\frac{|\bx|^2}{4}\right),
 \qquad \bx\in[-\pi,\pi)^d.
\]
Hence,
\[
 E_B\subset\{\bx\in[-\pi,\pi)^d:|\bx|^2\le4\log B\}.
\]
If $\omega_d$ denotes the volume of the Euclidean unit ball in $\R^d$, then
\[
 \mu_d(E_B)
 \le\frac{\omega_d(2\sqrt{\log B})^d}{(2\pi)^d}
 =\omega_d\left(\frac{\sqrt{\log B}}{\pi}\right)^d.
\]
Using
\[
 \omega_d=\frac{\pi^{d/2}}{\Gamma(d/2+1)}
 \le\left(\frac{2\pi e}{d}\right)^{d/2},
\]
we obtain
\begin{equation}\label{eq:EB-measure}
 \mu_d(E_B)\le
 \left(\frac{2e\log B}{\pi d}\right)^{d/2}.
\end{equation}
Combining \eqref{4-3} and \eqref{eq:EB-measure} proves \eqref{eq:HC-lower}.
\end{proof}

As an immediate consequence, if
\[
 |W|\le C(d)\bigl(\dim\cT(1,d)\bigr)^2,
\]
then Theorem~\ref{thm:HC-obstruction} implies
\begin{equation}\label{eq:B-growth}
 \log B\ge\frac{\pi d}{162e\,C(d)^{2/d}}.
\end{equation}
In particular, if $C(d)$ is bounded by a polynomial in $d$, then $B\ge e^{cd}$ for some absolute constant $c>0$ and all sufficiently large $d$.

\section{A transfer principle for random sampling discretization of the uniform norm}

Let $\xi^1,\xi^2,\ldots$ be independent random points distributed according to $\mu_d$ on $\T^d$, and set
\[
 X_m:=\{\xi^1,\ldots,\xi^m\},
 \qquad m\in\N.
\]

We first prove a general transfer principle for finite-dimensional translation-invariant spaces.

\begin{theorem}\label{thm:transfer}
Let $V$ be a translation-invariant linear subspace of $C(\T^d)$ with
\[
 \dim_{\R}V=n_0<\infty.
\]
Assume that there exist a constant $B\ge2$ and a finite set $W\subset\T^d$, with $|W|=m_0$, such that
\begin{equation}\label{eq:transfer-deterministic}
 \norm f_\infty\le B\norm f_W,
 \qquad f\in V.
\end{equation}
Then, for every $\eta\in(0,1)$ and every integer
\begin{equation}\label{eq:transfer-m}
 m\ge m_0\left(n_0\log(5B)+\log\frac1\eta\right),
\end{equation}
one has, with probability at least $1-\eta$,
\[
 \norm f_\infty\le4B\norm f_{X_m},
 \qquad f\in V.
\]
\end{theorem}

For $B>1$ and $\eta\in(0,1)$, let $m_*(N,d,B,\eta)$ denote the least positive  integer $m$ such that
\[
 \norm f_\infty\le B\norm f_{X_m},
 \qquad f\in\cT(N,d),
\]
holds simultaneously  with probability at least $1-\eta$.
Combining Theorem~\ref{thm:transfer} with the deterministic construction in Theorem~\ref{thm-1-1}   gives the following preliminary random-sampling bound.

\begin{corollary}\label{cor:N2}
Let $d,N\ge2$, $\varepsilon\in(0,1)$, and $\eta\in(0,1)$. There exist constants $B_{\varepsilon,d}>0$ and $C(d,\varepsilon)>0$ such that
\[
 m_*\bigl(N,d,B_{\varepsilon,d},\eta\bigr)
 \le C(d,\varepsilon)N^{2+\varepsilon}
 \left(1+\log\frac1\eta\right).
\]
\end{corollary}

\begin{proof}
Apply Theorem~\ref{thm-1-1} with $\varepsilon/2$ in place of $\varepsilon$. The resulting deterministic norming set has cardinality
\[
 m_0\le C(d,\varepsilon)N^{1+\varepsilon/2},
\]
and its norming constant depends only on $d$ and $\varepsilon$. Moreover,
\[
 n_0:=\dim_{\R}\cT(N,d)
 \le C_dN\log^{d-1}(2N).
\]
Theorem~\ref{thm:transfer} therefore applies once
\[
 m\ge C(d,\varepsilon)N^{1+\varepsilon/2}
 \left(N\log^{d-1}(2N)+\log\frac1\eta\right).
\]
Since $\log^{d-1}(2N)\le C(d,\varepsilon)N^{\varepsilon/2}$ for $N\ge2$, the asserted estimate follows.
\end{proof}

Now we return to the proof of Theorem \ref{thm:transfer}.
For $B>1$ and $f\in C(\T^d)$, define the superlevel set
\begin{equation}\label{5-1eq}
 S_f(B):=\left\{\bx\in\T^d:\
 |f(\bx)|\ge\frac1B\norm f_\infty\right\}.
\end{equation}
The following elementary observation is the main link between deterministic and random norming.

\begin{lemma}\label{thm:remez}
Let $V\subset C(\T^d)$ be translation invariant. Suppose that a finite set $W\subset\T^d$ and a constant $B>1$ satisfy
\begin{equation}\label{5-3a}
 \norm f_\infty\le B\norm f_W,
 \qquad f\in V.
\end{equation}
Then
\begin{equation}\label{5-2}
 \mu_d\bigl(S_f(B)\bigr)\ge\frac1{|W|},
 \qquad f\in V.
\end{equation}
\end{lemma}

\begin{proof}
Fix $f\in V$. For every $\bfa\in\T^d$, translation invariance and \eqref{5-3a} give
\[
 \norm f_\infty
 =\norm{f(\cdot-\bfa)}_\infty
 \le B\norm{f(\cdot-\bfa)}_W.
\]
Thus, for every $\bfa\in\T^d$, there exists $\bw\in W$ such that
\[
 |f(\bw-\bfa)|\ge B^{-1}\norm f_\infty.
\]
Equivalently, $\bfa\in\bw-S_f(B)$, and therefore
\[
 \T^d=\bigcup_{\bw\in W}\bigl(\bw-S_f(B)\bigr).
\]
Taking Haar measure proves \eqref{5-2}.
\end{proof}

\begin{remark}
Lemma~\ref{thm:remez} immediately yields a Nikolskii inequality. Indeed, for $0<p<\infty$ and $f\in V$,
\[
 \norm f_p^p
 \ge\int_{S_f(B)}|f(\bx)|^p\,d\mu_d(\bx)
 \ge\frac1{|W|}\left(\frac{\norm f_\infty}{B}\right)^p.
\]
Consequently,
\[
 \norm f_\infty\le B|W|^{1/p}\norm f_p,
 \qquad f\in V.
\]
\end{remark}

\begin{remark}
Applying Lemma~\ref{thm:remez} to the norming set supplied by Theorem~\ref{thm-1-1}, we obtain, for every $f\in\cT(N,d)$ and $\varepsilon\in(0,1)$,
\[
 \mu_d\bigl(S_f(B_\varepsilon)\bigr)
 \ge\frac1{A_\varepsilon N^{1+\varepsilon}},
\]
where
\begin{equation}\label{5-3}
 B_\varepsilon:=\bigl(C_2(1+\varepsilon^{-1})\bigr)^d,
 \qquad
 A_\varepsilon:=\bigl(C_1(1+\varepsilon^{-1})\bigr)^d.
\end{equation}
\end{remark}

\begin{proof}[Proof of Theorem~\ref{thm:transfer}]
Let
\[
 \delta:=\frac1{2B}
 \qquad\text{and}\qquad
 \mathcal B:=\{f\in V:\norm f_\infty\le1\}.
\]
Choose a maximal $\delta$-separated subset $\mathcal N$ of $\mathcal B$ in the uniform norm. Since $V$ has real dimension $n_0$, the standard volumetric estimate gives
\[
 |\mathcal N|
 \le\left(1+\frac2\delta\right)^{n_0}
 =(1+4B)^{n_0}
 \le(5B)^{n_0}.
\]
By maximality, $\mathcal N$ is also a $\delta$-net of $\mathcal B$. Furthermore, by Lemma~\ref{thm:remez},
\[
 \mu_d(S_g(B))\ge\frac1{m_0},
 \qquad g\in\mathcal N.
\]

For a nonzero $g\in V$,
\begin{equation}\label{5-6}
 X_m\cap S_g(B)\ne\varnothing
 \quad\Longleftrightarrow\quad
 \norm g_\infty\le B\norm g_{X_m}.
\end{equation}
Let $E$ be the event that $X_m\cap S_g(B)\ne\varnothing$ for every $g\in\mathcal N$.
Then
\[ E^c= \bigcup_{g\in\mathcal N} \Bl\{ X_m\cap S_g(B)=\varnothing \Br\}=\bigcup_{g\in\mathcal N} \Bl\{ \xi^j\notin S_g(B),\  \ \forall 1\leq j\leq m\Br\}.\]
Therefore, independence, the union bound, and \eqref{eq:transfer-m}  imply
\begin{align*}
 \Prb(E^c)
 &\le |\mathcal N|\sup_{g\in\mathcal N}
 \bigl(1-\mu_d(S_g(B))\bigr)^m\le(5B)^{n_0}\left(1-\frac1{m_0}\right)^m\\
 &\le\exp\left(n_0\log(5B)-\frac{m}{m_0}\right)
 \le\eta.
\end{align*}

Assume that $E$ occurs, and let $f\in V$ with $\norm f_\infty=1$. Choose $g\in\mathcal N$ such that $\norm{f-g}_\infty\le\delta$. Then $\norm g_\infty\ge1-\delta$, while \eqref{5-6} gives $\norm g_\infty\le B\norm g_{X_m}$. Hence
\begin{align*}
 \norm f_{X_m}
 &\ge\norm g_{X_m}-\norm{f-g}_\infty\\
 &\ge\frac{1-\delta}{B}-\delta
 =\frac{B-1}{2B^2}
 \ge\frac1{4B}.
\end{align*}
By homogeneity,
\[
 \norm f_\infty\le4B\norm f_{X_m},
 \qquad f\in V.
\]
Since $\Prb(E)\ge1-\eta$, the proof is complete.
\end{proof}

\section{Random norming for commuting regular partitions}\label{sec:abstract}

In this section we prove an abstract random norming theorem that will be used in the proof of Theorem~\ref{thm:main}. The result applies to sums of spaces of functions that are constant on regular partitions of a finite probability space, under the assumption that the corresponding conditional expectation operators commute. The proof combines a joint orthogonal decomposition, a matrix Chernoff estimate, and a polynomial that annihilates the average of the conditional expectations.

\subsection{General setting and main result}

Let $\Omega$ be a finite set equipped with the uniform normalised counting measure $\mu$; thus,
\[
 \mu(\{\omega\})=\frac1{|\Omega|},
 \qquad \omega\in\Omega.
\]
Throughout this section, $\langle\cdot,\cdot\rangle$ and $\norm{\cdot}_2$ denote the inner product and norm of $L^2(\Omega,\mu)$, respectively, and
\[
 \norm F_\infty:=\max_{x\in\Omega}|F(x)|.
\]

A partition $\mathcal P=\{E_1,\ldots,E_s\}$ of $\Omega$ is called \emph{regular} if all its cells have the same measure. Equivalently,
\[
 \mu(E_j)=\frac1s,
 \qquad
 |E_j|=\frac{|\Omega|}{s},
 \qquad 1\le j\le s.
\]

Let $\mathcal P_1,\ldots,\mathcal P_K$ be regular partitions of $\Omega$; repetitions are permitted. For $1\le j\le K$, set
\[
 r_j:=|\mathcal P_j|,
\]
let $\mathcal F_j$ be the $\sigma$-algebra generated by $\mathcal P_j$, and let
\[
 P_j:=\EE[\,\cdot\mid\mathcal F_j]
\]
be the associated conditional expectation. Then $P_j$ is the orthogonal projection of $L^2(\Omega,\mu)$ onto the space
\[
 H_j:=\left\{\sum_{A\in\mathcal P_j}c_A\mathbf1_A:c_A\in\C\right\}=\operatorname{span}\Bl\{\1_A:\  \ A\in \cP_j\Br\}.
\]
More explicitly,
\[
 P_jF(x)=\sum_{A\in\mathcal P_j}
 \left(r_j\int_A F\,d\mu\right)\mathbf1_A(x).
\]

Set
\[
 r:=\max_{1\le j\le K}r_j,
 \qquad
 V:=H_1+\cdots+H_K,
 \qquad
 S:=\frac1K\sum_{j=1}^KP_j,
\]
and assume that $r\ge2$ and
\begin{equation}\label{eq:commute}
 P_iP_j=P_jP_i,
 \qquad 1\le i,j\le K.
\end{equation}
We shall use the constants
\[
 M_K:=\binom{2K}{K},
 \qquad
 C_K:=2^{K+1}M_K.
\]

Let $u_1,u_2,\ldots$ be independent random points distributed uniformly on $\Omega$. For $m\in\N$, write
\[
 \Omega_m:=\{u_1,\ldots,u_m\},
 \qquad
 \norm F_{\Omega_m}:=\max_{1\le i\le m}|F(u_i)|.
\]
The sample points need not be distinct; multiplicities will always be retained in empirical sums. The main result of this section is the following.

\begin{theorem}\label{thm:abstract}
Let $\theta\in(0,1]$, set $\ell:=\lceil3/\theta\rceil$, and let $\eta\in(0,1)$. If
\begin{equation}\label{eq:samplesize}
 m\ge3C_K^2K^2r^{1+\theta}
 \log\left(\frac{2Kr}{\eta}\right),
\end{equation}
then, with probability at least $1-\eta$,
\[
 \norm F_\infty
 \le B(K,\theta)\norm F_{\Omega_m},
 \qquad F\in V,
\]
where
\[
 B(K,\theta):=2\ell M_K^\ell.
\]
Moreover,
\[
 B(K,\theta)\le\exp\left(\frac{CK}{\theta}\right)
\]
for an absolute constant $C>0$.
\end{theorem}

\subsection{The joint orthogonal decomposition}

For $\omega=(\omega_1,\ldots,\omega_K)\in\{0,1\}^K$, define
\[
 \Pi_\omega:=\prod_{j=1}^K
 \bigl(\omega_jP_j+(1-\omega_j)(I-P_j)\bigr),\  \ \text{and}\  \
 E_\omega:=\Bl\{\Pi_\omega f:\  \ f\in L^2(\Og, \mu)\Br\}.
\]
Note that
\[ \Pi_\og=\Bl(\prod_{j: \  \og_j=1} P_j\Br)\Bl( \prod_{j: \  \og_j=0} (I-P_j)\Br).\]
Since the factors are commuting orthogonal projections, $\Pi_\omega$ is an orthogonal projection. Furthermore, if $\omega\ne\omega'$, then there is an index $j$ for which one of the products contains $P_j$ and the other contains $I-P_j$. Since $P_j(I-P_j)=0$, it follows that
\[
 \Pi_\omega\Pi_{\omega'}=0,\  \ \forall \og\neq \og'.
\]
On the other hand, expanding
\[
 \prod_{j=1}^K\bigl(P_j+(I-P_j)\bigr)=I
\]
gives
\[
 \sum_{\omega\in\{0,1\}^K}\Pi_\omega=I.
\]
It follows that
\begin{equation}\label{eq:jointdiag}
 L^2(\Omega,\mu)
 =\bigoplus_{\omega\in\{0,1\}^K}E_\omega,
\end{equation}
and every $F\in L^2(\Omega,\mu)$ has the orthogonal expansion
\[
 F=\sum_{\omega\in\{0,1\}^K}\Pi_\omega F.
\]

For $1\le j\le K$, we have
\[
 P_j\Pi_\omega=\Pi_\omega P_j=\omega_j\Pi_\omega.
\]
Consequently,
\begin{equation}\label{6-4}
 P_jF
 =\sum_{\omega\in\{0,1\}^K}\omega_j\Pi_\omega F
 =\sum_{\omega:\,\omega_j=1}\Pi_\omega F,
\end{equation}
and hence
\begin{equation}\label{6-5-0}
 H_j=\bigoplus_{\omega:\,\omega_j=1}E_\omega.
\end{equation}
Averaging \eqref{6-4} over $j$ yields
\begin{equation}\label{6-6-0}
 S=\sum_{\omega\in\{0,1\}^K\setminus\{0\}}
 \frac{|\omega|_1}{K}\Pi_\omega.
\end{equation}
Therefore, for every algebraic polynomial $U$ and every $F\in V$,
\begin{equation}\label{6-5}
 U(S)F
 =\sum_{\omega\in\{0,1\}^K\setminus\{0\}}
 U\left(\frac{|\omega|_1}{K}\right)\Pi_\omega F.
\end{equation}
Finally, \eqref{6-5-0} gives
\[
 V=\bigoplus_{\omega\in\{0,1\}^K\setminus\{0\}}E_\omega.
\]
Thus, if $P_V$ denotes the orthogonal projection of $L^2(\Omega,\mu)$ onto $V$, then
\begin{equation}\label{6-9}
 P_V=\sum_{\omega\in\{0,1\}^K\setminus\{0\}}\Pi_\omega.
\end{equation}

\subsection{Lemmas}

We first construct a polynomial that annihilates $S$ on $V$.

\begin{lemma}\label{lem:annih}
Define
\[
 h_K(z):=\prod_{j=1}^K\left(1-\frac Kj z\right)
 =\sum_{j=0}^Kb_jz^j.
\]
Then the following statements hold.
\begin{enumerate}[\rm (i)]
 \item There is a polynomial $Q_K(z)=\sum_{j=0}^{K-1}c_jz^j$ such that
 \[
  1-h_K(z)=zQ_K(z),
  \qquad c_j=-b_{j+1},\quad 0\le j\le K-1.
 \]
 \item One has
 \[
  \sum_{j=0}^K|b_j|=M_K=\binom{2K}{K}\le4^K,
  \qquad
  \sum_{j=0}^{K-1}|c_j|\le M_K.
 \]
 \item $h_K(S)F=0$ for every $F\in V$.
\end{enumerate}
\end{lemma}

\begin{proof}
Since $h_K(0)=1$, the polynomial $1-h_K$ is divisible by $z$, and comparison of coefficients gives $c_j=-b_{j+1}$. This proves (i).

The coefficients of $h_K$ alternate in sign. Therefore,
\begin{align*}
 \sum_{j=0}^K|b_j|
 &=h_K(-1)
 =\prod_{s=1}^K\left(1+\frac Ks\right)\\
 &=\frac{(K+1)(K+2)\cdots(2K)}{K!}
 =\binom{2K}{K}
 \le4^K.
\end{align*}
It follows also that
\[
 \sum_{j=0}^{K-1}|c_j|
 =\sum_{j=1}^K|b_j|
 =M_K-1\le M_K.
\]
This proves (ii).

To prove (iii), apply \eqref{6-5} with $U=h_K$. If $\omega\ne0$, then $|\omega|_1\in\{1,\ldots,K\}$, and hence
\[
 h_K\left(\frac{|\omega|_1}{K}\right)=0.
\]
Thus $h_K(S)F=0$ for every $F\in V$.
\end{proof}

The next lemma provides the $(2,\infty)$-Nikolskii inequality for the space $V$.

\begin{lemma}\label{lem:leverage}
Let $K_V$ be the reproducing kernel of $V$ in $L^2(\Omega,\mu)$. Then
\begin{equation}\label{6-10a}
 K_V(x,x)\le Kr,
 \qquad x\in\Omega,
 \qquad\text{and}\qquad
 \dim V\le Kr.
\end{equation}
In particular,
\begin{equation}\label{8-4-Nik}
 \norm F_\infty\le\sqrt{Kr}\,\norm F_2,
 \qquad F\in V.
\end{equation}
\end{lemma}

\begin{proof}
We first claim that
\begin{equation}\label{6-11a}
 \langle P_VF,F\rangle
 \le\sum_{j=1}^K\langle P_jF,F\rangle,
 \qquad F\in L^2(\Omega,\mu).
\end{equation}
Indeed, by \eqref{6-4} and \eqref{6-9},
\begin{align*}
 \langle P_VF,F\rangle
 &=\sum_{\omega\ne0}\norm{\Pi_\omega F}_2^2\\
 &\le\sum_{j=1}^K\sum_{\omega:\,\omega_j=1}
 \norm{\Pi_\omega F}_2^2
 =\sum_{j=1}^K\langle P_jF,F\rangle.
\end{align*}

The reproducing kernel of $H_j$ is
\[
 K_{H_j}(x,y)
 =r_j\sum_{A\in\mathcal P_j}\mathbf1_A(x)\mathbf1_A(y),
\]
and therefore
\begin{equation}\label{6-12-a}
 K_{H_j}(x,x)=r_j,
 \qquad x\in\Omega.
\end{equation}
For $u\in\Omega$, define $e_u:=|\Omega|\mathbf1_{\{u\}}$. Then
\[
 \langle F,e_u\rangle=F(u),
 \qquad F\in L^2(\Omega,\mu),
\]
and
\[
 K_V(u,u)=\langle P_Ve_u,e_u\rangle.
\]
Using \eqref{6-11a} and \eqref{6-12-a}, we obtain
\[
 K_V(u,u)
 \le\sum_{j=1}^K\langle P_je_u,e_u\rangle
 =\sum_{j=1}^Kr_j
 \le Kr.
\]
This proves the first assertion in \eqref{6-10a}.

Let $\varphi_1,\ldots,\varphi_n$ be an orthonormal basis of $V$, where $n=\dim V$. Since
\[
 K_V(x,y)=\sum_{s=1}^n\varphi_s(x)\overline{\varphi_s(y)},
\]
we have
\[
 n=\int_\Omega K_V(x,x)\,d\mu(x)\le Kr.
\]
Finally, the reproducing property and the Cauchy--Schwarz inequality give
\begin{align*}
 |F(x)|
 &\le\norm F_2
 \left(\int_\Omega|K_V(x,y)|^2\,d\mu(y)\right)^{1/2}\\
 &=\norm F_2\sqrt{K_V(x,x)}
 \le\sqrt{Kr}\,\norm F_2.
\end{align*}
Taking the maximum over $x\in\Omega$ proves \eqref{8-4-Nik}.
\end{proof}

For $F,G:\Omega\to\C$, define
\[
 \langle F,G\rangle_m
 :=\frac1m\sum_{i=1}^mF(u_i)\overline{G(u_i)},
 \qquad
 \norm F_{L^2(\Omega_m)}:=\langle F,F\rangle_m^{1/2}.
\]
We first record the empirical $L^2$ estimate needed below, which follows directly from Lemma \ref{lem:leverage} and the Matrix Chernoff inequality \cite[Corollary~5.2]{Tropp} (see, for instance, \cite{DPTT} and \cite[Theorem 3.2]{DKT}).

\begin{lemma}[Matrix Chernoff estimate]\label{lem:chernoff}
Let $0<\delta\le1/2$, $\eta\in(0,1)$, and set $R:=Kr$. If
\begin{equation}\label{eq:mchernoff}
 m\ge3R\delta^{-2}\log(2R/\eta),
\end{equation}
then, with probability at least $1-\eta$,
\begin{equation}\label{eq:frame}
 (1-\delta)\norm F_2^2
 \le\norm F_{L^2(\Omega_m)}^2
 \le(1+\delta)\norm F_2^2,
 \qquad F\in V.
\end{equation}
\end{lemma}

%

We refer to the event in \eqref{eq:frame} as the \emph{frame event} with constant $\delta$. On this event,
\begin{equation}\label{6-16a}
 \bigl|\langle F,G\rangle_m-\langle F,G\rangle\bigr|
 \le\delta\norm F_2\norm G_2,
 \qquad F,G\in V.
\end{equation}
Indeed, let $\mathcal G_m:V\to V$ be determined by
\[
 \langle\mathcal G_mF,G\rangle=\langle F,G\rangle_m,
 \qquad F,G\in V.
\]
Then \eqref{eq:frame} is equivalent to
\[
 \norm{\mathcal G_m-I_V}_{2\to2}\le\delta,
\]
which immediately implies \eqref{6-16a}.

\begin{lemma}\label{lem:hit}
If the frame event \eqref{eq:frame} occurs, then every cell of every partition $\mathcal P_j$ contains at least one of the sample points $u_1,\ldots,u_m$.
\end{lemma}

\begin{proof}
Let $A\in\mathcal P_j$. Since $\mathbf1_A\in H_j\subset V$ and $\norm{\mathbf1_A}_2^2=\mu(A)>0$, the lower estimate in \eqref{eq:frame} gives
\[
 \frac1m\sum_{i=1}^m\mathbf1_A(u_i)
 =\norm{\mathbf1_A}_{L^2(\Omega_m)}^2
 \ge(1-\delta)\mu(A)>0.
\]
Thus $u_i\in A$ for at least one $i$.
\end{proof}

For the remainder of this subsection, we fix the random points $u_1, \ldots, u_m$ for which  the frame event occurs. For $A\subset\Omega$, set
\[
 N_A:=\bigl|\{1\le i\le m:u_i\in A\}\bigr|.
\]
By Lemma~\ref{lem:hit}, $N_A\ge1$ whenever $A$ is a cell of one of the partitions $\mathcal P_j$. We may therefore define
\[
 \widehat P_jF
 :=\sum_{A\in\mathcal P_j}
 \left(\frac1{N_A}\sum_{i:\,u_i\in A}F(u_i)\right)\mathbf1_A,
 \qquad 1\le j\le K,
\]
and
\[
 \widehat S:=\frac1K\sum_{j=1}^K\widehat P_j.
\]
We regard $\widehat P_j$ and $\widehat S$ as operators on $V$.

\begin{lemma}\label{lem:emp}
Assume that the frame event \eqref{eq:frame} occurs with $0<\delta\le1/2$. Then the following statements hold.
\begin{enumerate}[\rm (i)]
 \item The form $\langle\cdot,\cdot\rangle_m$ is an inner product on $V$, and $\widehat P_j|_V$ is the orthogonal projection of $(V,\langle\cdot,\cdot\rangle_m)$ onto $H_j$.
 \item For every $F\in V$,
 \[
  \norm{\widehat P_jF}_\infty\le\norm F_{\Omega_m},
  \qquad
  \norm{\widehat SF}_\infty\le\norm F_{\Omega_m}.
 \]
 Consequently,
 \[
  \norm{\widehat P_j}_{\infty\to\infty}\le1,
  \qquad
  \norm{\widehat S}_{\infty\to\infty}\le1.
 \]
 \item For every $F\in V$,
 \begin{equation}\label{6-17c}
  \norm{(\widehat P_j-P_j)F}_2\le2\delta\norm F_2.
 \end{equation}
 Consequently,
 \[
  \norm{\widehat S-S}_{2\to2}\le2\delta,
  \qquad
  \norm{\widehat S}_{2\to2}\le2.
 \]
\end{enumerate}
\end{lemma}

\begin{proof}
The lower estimate in \eqref{eq:frame} shows that $\langle\cdot,\cdot\rangle_m$ is positive definite on $V$. Hence, $\langle\cdot,\cdot\rangle_m$  is an inner product on $V$.  For each $A\in \cP_j$, let $\phi_A=\sqrt{\f m {N_A}}\1_A$. A straightforward calculation shows that $\{ \phi_A:\  A\in\cP_j\}$
is an orthonormal basis of $(H_j, \langle\cdot,\cdot\rangle_m)$, and
\[ \widehat P_jF= \sum_{A\in\cP_j} \langle F,\phi_A\rangle_m\phi_A\in H_j.\]
Hence, $\widehat P_j|_V$ is the orthogonal projection of $(V,\langle\cdot,\cdot\rangle_m)$ onto $H_j$.
This proves (i).

If $x\in A\in\mathcal P_j$, then
\[
 |\widehat P_jF(x)|
 =\left|\frac1{N_A}\sum_{i:\,u_i\in A}F(u_i)\right|
 \le\norm F_{\Omega_m}.
\]
Taking the maximum over $x$ proves the first estimate in (ii), and averaging over $j$ gives the estimate for $\widehat S$. Since $\norm F_{\Omega_m}\le\norm F_\infty$, the operator-norm estimates follow.

It remains to prove (iii). Put
\[
 G:=(\widehat P_j-P_j)F\in H_j.
\]
The lower frame estimate gives
\begin{equation}\label{6-18b}
 (1-\delta)\norm G_2^2
 \le\norm G_{L^2(\Omega_m)}^2.
\end{equation}
Write
\[ G=(\widehat P_j F-F) +(F-P_jF).\]
Since $F-\widehat P_jF$ is orthogonal to $H_j$ with respect to $\langle\cdot,\cdot\rangle_m$, while $F-P_jF$ is orthogonal to $H_j$ in $L^2(\Omega,\mu)$, we have
\begin{align}
 \norm G_{L^2(\Omega_m)}^2
 &=\langle F-P_jF,G\rangle_m\notag\\
 &=\langle F-P_jF,G\rangle_m
   -\langle F-P_jF,G\rangle.
 \label{6-18}
\end{align}
Using \eqref{6-16a} and the fact that $P_j$ is an orthogonal projection, we obtain
\[
 \norm G_{L^2(\Omega_m)}^2
 \le\delta\norm{F-P_jF}_2\norm G_2
 \le\delta\norm F_2\norm G_2.
\]
Together with \eqref{6-18b}, this gives
\[
 \norm G_2=\norm{(\widehat P_j-P_j)F}
 \le\frac{\delta}{1-\delta}\norm F_2
 \le2\delta\norm F_2,
\]
which proves \eqref{6-17c}. Averaging over $j$ yields
\[
 \norm{\widehat S-S}_{2\to2}\le2\delta.
\]
Finally, since $\norm S_{2\to2}\le1$, we have
\[
 \norm{\widehat S}_{2\to2}
 \le\norm S_{2\to2}+\norm{\widehat S-S}_{2\to2}
 \le1+2\delta\le2.
\]
\end{proof}

The next lemma transfers the closeness of $\widehat S$ and $S$ to the annihilating polynomial.

\begin{lemma}\label{lem:poly}
Assume that the frame event \eqref{eq:frame} occurs with $0<\delta\le1/2$, and let
\[
 \Lambda:=h_K(\widehat S)
\]
as an operator on $V$. Then
\[
 \norm\Lambda_{2\to2}\le C_K\delta,
 \qquad
 \norm\Lambda_{\infty\to\infty}\le M_K,
 \qquad
 \norm{Q_K(\widehat S)}_{\infty\to\infty}\le M_K.
\]
\end{lemma}

\begin{proof}
Set $A:=\widehat S$ and $B:=S$, regarded as operators on $V$. By Lemma~\ref{lem:emp},
\[
 \norm A_{2\to2}\le2,
 \qquad
 \norm B_{2\to2}\le1,
 \qquad
 \norm{A-B}_{2\to2}\le2\delta.
\]
By Lemma~\ref{lem:annih}(iii), $h_K(B)=0$ on $V$. Therefore,
\begin{equation}\label{6-21a}
 \Lambda=h_K(A)-h_K(B)
 =\sum_{j=1}^Kb_j(A^j-B^j).
\end{equation}
For $1\le j\le K$, the identity
\[
 A^j-B^j
 =\sum_{i=0}^{j-1}A^{j-i-1}(A-B)B^i
\]
gives
\[
 \norm{A^j-B^j}_{2\to2}
 \le2\delta\sum_{i=0}^{j-1}2^{j-i-1}
 \le2^{j+1}\delta.
\]
Substituting this estimate into \eqref{6-21a} and using Lemma~\ref{lem:annih}(ii), we obtain
\[
 \norm\Lambda_{2\to2}
 \le\delta\sum_{j=1}^K2^{j+1}|b_j|
 \le2^{K+1}M_K\delta
 =C_K\delta.
\]

By Lemma~\ref{lem:emp}(ii), $\norm A_{\infty\to\infty}\le1$. It follows that
\[
 \norm\Lambda_{\infty\to\infty}
 \le\sum_{j=0}^K|b_j|\norm A_{\infty\to\infty}^j
 \le M_K,
\]
and similarly,
\[
 \norm{Q_K(A)}_{\infty\to\infty}
 \le\sum_{j=0}^{K-1}|c_j|\norm A_{\infty\to\infty}^j
 \le M_K.
\]
\end{proof}

The final auxiliary lemma converts the small $L^2$ norm of $\Lambda$ into a strict contraction in the uniform norm after a bounded number of iterations.

\begin{lemma}\label{lem:amp}
Let $\theta\in(0,1]$, set $\ell:=\lceil3/\theta\rceil$, and assume that the frame event \eqref{eq:frame} occurs with
\[
 0<\delta\le(C_K\sqrt K)^{-1}r^{-\theta/2}.
\]
If $\Lambda:=h_K(\widehat S)$ on $V$, then
\[
 \norm{\Lambda^\ell}_{\infty\to\infty}
 \le r^{(1-\theta\ell)/2}
 \le r^{-1}
 \le\frac12.
\]
\end{lemma}

\begin{proof}
Let $F\in V$. Lemmas~\ref{lem:leverage} and \ref{lem:poly} give
\[
 \norm{\Lambda^\ell F}_\infty
 \le\sqrt{Kr}\,\norm{\Lambda^\ell F}_2
 \le\sqrt{Kr}\,(C_K\delta)^\ell\norm F_2.
\]
By the assumption on $\delta$,
\[
 C_K\delta\le K^{-1/2}r^{-\theta/2}.
\]
Since $\mu$ is a probability measure, it follows that
\begin{align*}
 \norm{\Lambda^\ell F}_\infty
 &\le K^{(1-\ell)/2}r^{(1-\theta\ell)/2}\norm F_\infty\le r^{(1-\theta\ell)/2}\norm F_\infty.
\end{align*}
Here we used $\ell\ge3$ and $K\ge1$. Since $\theta\ell\ge3$ and $r\ge2$, we have
\[
 r^{(1-\theta\ell)/2}\le r^{-1}\le\frac12.
\]
\end{proof}
\begin{remark}
An important aspect of the proofs above  is that they do not require the empirical projections $\widehat P_j$ to commute. Their uniform-norm contractivity is sufficient. Commutativity is used for the original projections, to obtain $h_K(S)=0$. 
\end{remark}

\subsection{Proof of Theorem~\ref{thm:abstract}}

Set
\[
 \delta:=(C_K\sqrt K)^{-1}r^{-\theta/2}.
\]
Then $0<\delta\le1/2$. By Lemma~\ref{lem:chernoff}, the frame event with constant $\delta$ has probability at least $1-\eta$ provided that
\begin{align*}
 m
 &\ge 3Kr\,\delta^{-2}\log(2Kr/\eta)=3C_K^2K^2r^{1+\theta}\log(2Kr/\eta),
\end{align*}
which is precisely \eqref{eq:samplesize}.

Assume henceforth that the frame event occurs. First, recall that $h_K(S)=0$ on $V$, so we may write every $F\in V$ in the form  $F=(I-h_K(S))F$.
 Put
\[
 \Lambda:=h_K(\widehat S).
\]
Since $1-h_K(z)=zQ_K(z)$, Lemma~\ref{lem:annih}(i) gives, for every $F\in V$,
\begin{equation}\label{6-20}
 (I-\Lambda)F
 =\bigl(I-h_K(\widehat S)\bigr)F
 =Q_K(\widehat S)\widehat S F.
\end{equation}

We claim that $I-\Lambda$ is invertible on the normed space $(V,\norm\cdot_\infty)$ and that
\begin{equation}\label{6-21}
 \norm{(I-\Lambda)^{-1}}_{\infty\to\infty}
 \le 2\ell M_K^{\ell-1}.
\end{equation}

Assuming the claim for the moment, we obtain  from \eqref{6-20} that
\[
 F=(I-\Lambda)^{-1}Q_K(\widehat S)\widehat S F.
\]
This combined with \eqref{6-21} and Lemmas~\ref{lem:poly} and \ref{lem:emp}(ii), gives
\begin{align*}
 \norm F_\infty
 &\le 2\ell M_K^{\ell-1}
 \norm{Q_K(\widehat S)\widehat S F}_\infty\le 2\ell M_K^\ell\norm{\widehat S F}_\infty\le 2\ell M_K^\ell\norm F_{\Omega_m}.
\end{align*}
Since $F\in V$ was arbitrary, we  prove the asserted norming inequality with
\[
 B(K,\theta)=2\ell M_K^\ell.
\]
Moreover, since $M_K\le4^K$ and $\ell\le4/\theta$, we have
\[
 \log \, B(K,\theta)
 \le \log \Bl[\frac8\theta\,4^{4K/\theta}\Br]=\log\frac8\theta+\frac{4K\log4}{\theta}
 \le \frac{CK}{\theta},
\]
where $C>0$ is an absolute constant.

It remains to prove the claim \eqref{6-21}. By Lemma~\ref{lem:amp},
\[
 \norm{\Lambda^\ell}_{\infty\to\infty}\le\frac12.
\]
Hence $I-\Lambda^\ell$ is invertible on $(V,\norm\cdot_\infty)$, and the Neumann series gives
\[
 \norm{(I-\Lambda^\ell)^{-1}}_{\infty\to\infty}
 \le \sum_{j=0}^\infty
 \norm{\Lambda^\ell}_{\infty\to\infty}^{\,j}
 \le 2.
\]
Since
\[
 I-\Lambda^\ell
 =(I-\Lambda)(I+\Lambda+\cdots+\Lambda^{\ell-1}),
\]
and the two factors on the right-hand side commute, the operator $I-\Lambda$ is also invertible and
\[
 (I-\Lambda)^{-1}
 =(I+\Lambda+\cdots+\Lambda^{\ell-1})
 (I-\Lambda^\ell)^{-1}.
\]
By Lemma~\ref{lem:poly},
\[
 \norm{\Lambda}_{\infty\to\infty}\le M_K.
\]
Therefore,
\begin{align*}
 \norm{(I-\Lambda)^{-1}}_{\infty\to\infty}
 &\le
 \left(\sum_{s=0}^{\ell-1}
 \norm{\Lambda}_{\infty\to\infty}^{\,s}\right)
 \norm{(I-\Lambda^\ell)^{-1}}_{\infty\to\infty}\\
 &\le 2\sum_{s=0}^{\ell-1}M_K^s
 \le 2\ell M_K^{\ell-1}.
\end{align*}
This proves \eqref{6-21} and completes the proof of the theorem.

\section{The grids of $\T^d$  and the de la Vall\'ee Poussin decomposition}

Let $\ell_0\ge1$ be an integer to be chosen later. For $n\in\ZZ_+$ and
$0\le j<\ell_0 2^{n+1}$, set
\[
 I_{n,j}:=\left[\frac{2\pi j}{\ell_0 2^{n+1}},
 \frac{2\pi(j+1)}{\ell_0 2^{n+1}}\right),
\]
and define
\[
 \CD_n:=\Bl\{I_{n,j}:\  \ j=0,1,\ldots, \ell_0 2^{n+1}-1\Br\}.
\]
Thus, $\CD_n$ is a partition of $\T$ into $\ell_0 2^{n+1}$ intervals of equal length. Let $\FF_n$ denote the $\sigma$-algebra generated by $\CD_n$, and let
\[
 P_nf:=\EE[f\mid\FF_n]
 =\sum_{I\in\CD_n}\left(\frac1{\mu_1(I)}\int_I f\,d\mu_1\right)\1_I.
\]
Then $P_n$ is the orthogonal projection of $L^2(\T)$ onto the space of functions that are constant on each interval of $\CD_n$. Since $\CD_n$ refines $\CD_m$ whenever $m\le n$, we have
\begin{equation}\label{7-1a}
 P_nP_m=P_mP_n=P_{m\wedge n},
 \qquad m,n\in\ZZ_+.
\end{equation}

 For a multi-index $\bn=(n_1,\ldots,n_d)\in\ZZ_+^d$, define
\[
 \CD_{\bn}:=\{I_1\times\cdots\times I_d:\  I_j\in\CD_{n_j},\  1\leq j\leq d\}.
\]
Clearly, $\CD_{\bn}$ is a regular partition of $\T^d$ satisfying
\begin{equation}\label{7-2-b}
 |\CD_{\bn}|=(2\ell_0)^d2^{|\bn|_1}
\end{equation}
and
\[
 \mu_d(Q)=\frac1{(2\ell_0)^d2^{|\bn|_1}}
 =\frac1{|\CD_{\bn}|},
 \qquad Q\in\CD_{\bn}.
\]
For $\bm,\bn\in\ZZ_+^d$, we write $\bm\preceq\bn$ if $m_j\le n_j$ for all $1\le j\le d$. If $\bm\preceq\bn$, then $\CD_{\bn}$ refines $\CD_{\bm}$. Let $\FF_{\bn}$ be the $\sigma$-algebra generated by $\CD_{\bn}$, and put
\[
 P_{\bn}f:=\EE[f\mid\FF_{\bn}]
 =\sum_{Q\in\CD_{\bn}}\left(\frac1{\mu_d(Q)}\int_Qf\,d\mu_d\right)\1_Q.
\]
Since
\[
 P_{\bn}=P_{n_1}^{(1)}\cdots P_{n_d}^{(d)},
\]
it follows from \eqref{7-1a} that
\begin{equation}\label{7-3}
 P_{\bn}P_{\bm}=P_{\bm}P_{\bn}=P_{\bm\wedge\bn},
 \qquad \bm,\bn\in\ZZ_+^d,
\end{equation}
where
\[
 \bm\wedge\bn
 :=\bigl(\min\{m_1,n_1\},\ldots,\min\{m_d,n_d\}\bigr).
\]

We next introduce the dyadic de la Vall\'ee Poussin decomposition. For $f,g\in L^1(\T)$, write
\[
 (f*g)(x):=\int_{\T}f(x-t)g(t)\,d\mu_1(t).
\]
Let $K_n$ be the de la Vall\'ee Poussin kernel defined by (2.1). Its Fourier coefficients are given by
\begin{equation}\label{eq:Kn-Fourier}
 \widehat K_n(k)=
 \begin{cases}
  1, & |k|\le n,\\[2mm]
  2-\dfrac{|k|}{n}, & n<|k|<2n,\\[2mm]
  0, & |k|\ge2n.
 \end{cases}
\end{equation}
For $n\in\ZZ_+$, define the operator $\Psi_n$ by
\begin{equation}\label{eq:Psi-one-dimensional}
 \Psi_0f:=f*K_1,
 \qquad
 \Psi_nf:=f*(K_{2^n}-K_{2^{n-1}}),\quad n\ge1.
\end{equation}
For later use, set
\[
 J_0:=\{-1,0,1\},
 \qquad
 J_n:=\{k\in\ZZ:2^{n-1}<|k|<2^{n+1}\},\quad n\ge1.
\]
By \eqref{eq:Kn-Fourier}, the spectrum of $\Psi_nf$ is contained in $J_n$.

For $\bn=(n_1,\ldots,n_d)\in\ZZ_+^d$, define
\begin{equation}\label{eq:Psi-multi}
 \Psi_{\bn}f
 :=\Psi_{n_1}^{(1)}\cdots\Psi_{n_d}^{(d)}f,
 \qquad
 J_{\bn}:=J_{n_1}\times\cdots\times J_{n_d}.
\end{equation}
Throughout the remainder of this section, we assume that $N\ge8$ and set
\[
 L:=\lceil\log_2N\rceil,
 \qquad
 [L]_0:=\{0,1,\ldots,L\}.
\]
Thus, $L\ge3$ and $2^{L-1}<N\le2^L$.

The following lemma collects some well-known results in classical analysis.

\begin{lemma}\label{lem:dlvp}
The following statements hold.
\begin{enumerate}[\rm (i)]
 \item For every $f\in\cT([-N,N]^d)$,
 \[
  f=\sum_{\bn\in[L]_0^d}\Psi_{\bn}f.
 \]
 \item If $\bn\in\ZZ_+^d$ and $f\in L^1(\T^d)$, then $\Psi_{\bn}f$ is a trigonometric polynomial with spectrum contained in $J_{\bn}$. In particular,
 \[
  \norm{\partial_j\Psi_{\bn}f}_\infty
  \le2^{n_j+1}\norm{\Psi_{\bn}f}_\infty,
  \qquad 1\le j\le d.
 \]
 \item There exists an absolute constant $A\ge1$ such that, for every $1\le p\le\infty$,
 \[
  \sup_{\bn\in\ZZ_+^d}\norm{\Psi_{\bn}f}_p
  \le A^d\norm f_p,
  \qquad f\in L^p(\T^d).
 \]
 In fact, one can choose here  $A=6$.
\end{enumerate}
\end{lemma}

\begin{lemma}\label{lem:cellwise}
Let $\bn\in\ZZ_+^d$, and suppose that $\bx,\by$ belong to the same box of $\CD_{\bn}$. Then, for every $f\in C(\T^d)$,
\[
 \abs{\Psi_{\bn}f(\bx)-\Psi_{\bn}f(\by)}
 \le\frac{2\pi dA^d}{\ell_0}\norm f_\infty.
\]
\end{lemma}

\begin{proof}
Put $f_{\bn}:=\Psi_{\bn}f$, and define
\[
 \bz^0:=\bx,
 \qquad
 \bz^j:=(y_1,\ldots,y_j,x_{j+1},\ldots,x_d),
 \quad 1\le j\le d.
\]
Since $\bx$ and $\by$ belong to the same box of $\CD_{\bn}$,
\[
 |x_j-y_j|\le\frac{2\pi}{\ell_0 2^{n_j+1}},
 \qquad 1\le j\le d.
\]
Using Lemma~\ref{lem:dlvp}(ii) and (iii), we obtain
\begin{align*}
 |f_{\bn}(\bx)-f_{\bn}(\by)|
 &\le\sum_{j=1}^d
 |f_{\bn}(\bz^{j-1})-f_{\bn}(\bz^j)|
 \le\sum_{j=1}^d\norm{\partial_jf_{\bn}}_\infty|x_j-y_j|\\
 &\le\sum_{j=1}^d2^{n_j+1}\norm{f_{\bn}}_\infty
 \frac{2\pi}{\ell_0 2^{n_j+1}}
 \le\frac{2\pi dA^d}{\ell_0}\norm f_\infty.
\end{align*}
\end{proof}

If $f\in\cT(N,d)$ and $\Psi_{\bn}f$ is not identically zero, then
\[
 \Gamma(N,d)\cap J_{\bn}\ne\emptyset.
\]
The next lemma shows that all active indices belong to a thin simplex.

\begin{lemma}\label{lem:simplex}
If $\bn\in\ZZ_+^d$ and $\Gamma(N,d)\cap J_{\bn}\ne\emptyset$, then
\[
 |\bn|_1<L+d.
\]
\end{lemma}

\begin{proof}
Choose $\bk=(k_1,\ldots,k_d)\in\Gamma(N,d)\cap J_{\bn}$. If $n_j\ge1$, then
\[
 2^{n_j-1}<|k_j|<2^{n_j+1},
\]
whereas if $n_j=0$, then $|k_j|\le1$. Thus, in either case,
\[
 2^{n_j-1}<\max\{|k_j|,1\}.
\]
Multiplying these inequalities over $j=1,\ldots,d$ and using $\bk\in\Gamma(N,d)$, we obtain
\[
 2^{|\bn|_1-d}
 <\prod_{j=1}^d\max\{|k_j|,1\}
 \le N\le2^L.
\]
Hence $|\bn|_1<L+d$.
\end{proof}

Define
\[
 M_{L+d}:=\{\bn\in[L]_0^d:|\bn|_1<L+d\},
 \qquad
 S_{L+d}:=\{\bn\in[L]_0^d:|\bn|_1=L+d\}.
\]
For $f\in\cT(N,d)$, put
\[
 \cA(f):=\{\bn\in[L]_0^d:\Psi_{\bn}f\not\equiv0\}.
\]
Lemma~\ref{lem:simplex} implies
\begin{equation}\label{eq:active-simplex}
 \cA(f)\subset M_{L+d}.
\end{equation}

\begin{lemma}\label{lem-7-4}
For every $\bn\in M_{L+d}$, there exists $\bm\in S_{L+d}$ such that
$\bn\preceq\bm$.
\end{lemma}

\bp
Let   $\bn=(n_1,\ldots, n_d)\in M_{L+d}$.
Then   $0\leq n_j\leq L$ for  $1\leq j\leq d$,  and hence
\begin{align*}
\Bl\{ \bm\in [L]_0^d:\  \ \bn\prec \bm\Br\}&=
\Bl\{ \sum_{j=1}^d (n_j+k_j):\  \ k_j=0,1,\ldots, L-n_j,\  \ 1\leq j\leq d\Bl\} \\
& =\Bl\{ |\bn|_1,  |\bn|_1+1, \ldots, Ld\Br\}. \end{align*}
Since
$|\bn|_1< L+d <Ld$,
 there must exist  $\bm\in S_{L+d}$ such that $\bn\prec \bm$.
 \ep

\section{  Sampling discretization of uniform  norms  by random points}

The main goal of this section is to prove Theorem~\ref{thm:main}. We keep the notation of Section~7. Let
\[
 \bL:=(L,\ldots,L)\in\ZZ_+^d,
 \qquad
 \Gt:=\CD_{\bL}.
\]
Thus, $\Gt$ is a partition of $\T^d$ into
\[
 |\Gt|=(2\ell_0)^d2^{dL}=(\ell_02^{L+1})^d
\]
congruent boxes. Let $\Omega$ be the finite set consisting of the boxes in $\Gt$, equipped with the uniform normalised counting measure $\mu$. We will apply  Theorem~\ref{thm:abstract}  to  the probability  space  $(\Og, \mu)$.
Define
\[
 \pi:\T^d\longrightarrow\Omega
\]
by letting $\pi(\bx)$ be the unique box of $\Gt$ that contains $\bx$.  For each cell $Q\in \Og$ and any $\bx\in\T^d$, we have
\[ \bx\in Q\iff \pi(\bx)= Q,\]
which immediately gives
$$Q=\pi^{-1}(\{Q\}).$$
Since all boxes of $\Gt$ have the same normalized Lebesgue measure, it follows that
\[ \mu(\{Q\})=\f 1 {|\Og|}=\mu_d(Q)=\mu_d (\pi^{-1}(\{Q\})),\  \ \qquad Q\in\Omega.\]
Thus,  $\pi:\ \T^d\to \Og$ pushes normalised Lebesgue  measure $\mu_d$ on $\T^d$
forward to the uniform probability measure $\mu$ on $\Og$.
In particular,  if $\xi^1,\ldots,\xi^m$ are independent random points distributed according to $\mu_d$ on $\T^d$, then
\[
 u_i:=\pi(\xi^i),
 \qquad 1\le i\le m,
\]
are independent random points distributed uniformly on $\Omega$.

We call a partition $\Pt$ of $\T^d$ $\Gt$-coarse if each cell of $\Pt$ is a union of boxes from $\Gt$. In this case,  for each $A\in\Pt$, we denote by  $\mathcal{C}(A)$ the collection of all boxes $Q\in\Gt$ that are contained in $A$. By definition,  for each  $\bn\in [L]_0^d$, $\CD_{\bn}$ is a \emph{$\Gt$-coarse} partition of $\T^d$.
 If $\Pt$ is $\Gt$-coarse, define
\[
 \Pt^\Omega
 :=\bigl\{\mathcal{C}(A):\  A\in\Pt\bigr\}.
\]
Then $\Pt^\Omega$ is a partition of $\Omega$. Conversely, every partition of $\Omega$ determines, by taking unions of boxes in $\Gt$, a unique $\Gt$-coarse partition of $\T^d$. This correspondence preserves cardinality, refinement, and regularity.

If $g:\T^d\to\C$ is constant on the cells of a $\Gt$-coarse partition $\Pt$, we say  $g$ is  $\Pt$-constant  on $\T^d$. In this case,  there exists  a unique $\Pt^\Og$-constant function $g^\Omega:\Omega\to\C$ such that
\begin{equation}\label{8-1}
 g=g^\Omega\circ\pi.
\end{equation}
Moreover,
\[
 \norm g_{L^\infty(\T^d)}=\norm{g^\Omega}_{\ell^\infty(\Omega)},
 \qquad
 \norm g_{L^2(\T^d)}=\norm{g^\Omega}_{L^2(\Omega,\mu)}.
\]
In what follows, we identify a $\Gt$-coarse partition with the corresponding partition of $\Omega$, and we omit the superscript $\Omega$ whenever no confusion can arise.

\subsection{The macro-partitions}\label{sec-8-1}

Let $q$ be an integer satisfying
\[
 2\le q\le\sqrt{L+1}.
\]
The exact value of $q$ will be specified in the next subsection.  Put $h:=\Bigl\lceil\frac{L+1}{q}\Bigr\rceil$.
Since $q^2\leq L+1$, we have $$(q-1)h<L+1\leq qh.$$ Thus,  $[L]_0=\{0, 1,\ldots, L\}$ can be written as the disjoint union of $q$ consecutive intervals
\[
 I_1,\ldots,I_q
\]
such that
\[
 |I_1|=\cdots=|I_{q-1}|=h,
 \qquad
 |I_q|\le h.
\]
For a multi-index $\beta\in\{1,\dots,q\}^d$,  define $$C_\beta:=I_{\beta_1}\times\dots\times I_{\beta_d}\subset [L]_0^d.$$
The sets $C_\be$ form a partition of $[L]_0^d$.

We call   $\beta\in \{1,2,\ldots, q\}^d$
\emph{retained} if $C_\beta\cap S_{L+d}\ne\emptyset$, and  denote  the set of all retained indices by
\[\mathcal{I}:= \Bl\{ \be\in  \{1, 2,\ldots, q\}^d:\   \  C_\be \cap S_{L+d} \neq \emptyset\Br\}.\]
For $\beta\in\mathcal{I}$, define
\begin{equation}\label{eq:abeta}
 a(\beta)_j
 :=\max\{n_j:\bn=(n_1,\ldots,n_d)\in C_\beta\cap S_{L+d}\},
 \qquad 1\le j\le d,
\end{equation}
and put
\[
 a(\beta):=(a(\beta)_1,\ldots,a(\beta)_d)\in C_\be,
 \qquad
 \Qt_\beta:=\CD_{a(\beta)}.
\]
For every $\bn\in C_\beta\cap S_{L+d}$, we have $\bn\preceq a(\beta)$.

Let  $\be\in\mathcal{I}$. Recall that $\Gt=\CD_{(L,\ldots, L)}$.
Clearly, $\Qt_\beta$ is a $\Gt$-coarse partition of $\T^d$.
 We regard $\Qt_\beta$ also as the induced partition of $\Omega$. Let $H_\beta$ be the space of all $\Qt_\beta$-constant functions on $\Omega$, and let $P_\beta$ be the associated conditional expectation on $L^2(\Omega,\mu)$.

\begin{lemma}\label{lem:macro}
The following statements hold.
\begin{enumerate}[\rm (a)]
 \item For every $\beta\in\mathcal{I}$, the partition $\Qt_\beta$ is regular and $\Gt$-coarse.
 \item $|\mathcal{I}|\le dq^{d-1}$.
 \item The operators $P_\beta$, $\beta\in\mathcal{I}$, pairwise commute.
 \item For every $\tb\in M_{L+d}\cup S_{L+d}$, there exists $\beta\in\mathcal{I}$ such that
 \[
  \tb\preceq a(\beta).
 \]
 Consequently, $\Qt_\beta$ refines $\CD_{\tb}$.
 \item For every $\beta\in\mathcal{I}$,
 \[
  |\Qt_\beta|
  \le\ell_0^d2^{L+2d+d(h-1)}.
 \]
 Consequently,
 \begin{equation}\label{eq:macro-r}
  r:=\max_{\beta\in\mathcal{I}}|\Qt_\beta|
  \le C_d\ell_0^dN^{1+d/q}.
 \end{equation}
\end{enumerate}
\end{lemma}

\begin{proof}
Part (a) follows from the definition of $\Qt_\beta$ and the fact that $a(\beta)\preceq\bL$. 

(b)  For each  $\be\in \mathcal{I}$, there exists $\bn\in C_\be$ such that  $|\bn|_1=L+d$. Since 
\[ n_j\in I_{\be_j} \subset \Bl\{ (\be_j-1) h, (\be_j-1) h+1, \ldots, \be_j h-1\Br\},\]
we have
$$ (\be_j-1) h \leq n_j<\be_j h,\  \ 1\leq j\leq d.$$
Summing these inequalities  over $1\leq j\leq d$ gives 
\[ h (|\be|_1-d) \leq |\bn|_1 =L+d < |\be|_1 h.\]
Therefore, 
\[ \f {L+d} h <|\be|_1 \leq \f {L+d} h+d.\]
Define 
\[ A=\ZZ\cap \Bl( \f {L+d} h, \f {L+d} h +d\Br].\]
Then $|A|\leq d$ and the preceding estimate shows that $|\be|_1\in A$ for every $\be\in \mathcal{I}$. 
Thus, 
\[ \mathcal{I}\subset \Bl\{ (\be_1, \ldots, \be_{d-1}, \ell-\sum_{j=1}^{d-1}\be_j):\   \ \be_j\in [q],\  1\leq j\leq d-1,\  \ \ell\in A\Br\}.\]
It follows that 
$$|\mathcal{I}|\leq |A| q^{d-1}\leq d q^{d-1}. $$

Part (c) follows from \eqref{7-3}.

(d)\  \ Let $\tb\in M_{L+d}\cup S_{L+d}$. If $\tb\in M_{L+d}$, by Lemma~\ref{lem-7-4}, we may choose  $\sb\in S_{L+d}$ such that $\tb\preceq\sb$. If $\tb\in S_{L+d}$, we put $\sb:=\tb$. Since $\{C_\be:\  \ \be\in \{1,\ldots, q\}^d\}$ is a partition of $[L]_0^d$, there is a unique $\beta\in\{1,\ldots,q\}^d$ such that $\sb\in C_\beta\cap S_{L+d}$. Then $\beta\in\mathcal{I}$, and the definition of $a(\beta)$ gives
\[
 \tb\preceq\sb\preceq a(\beta).
\]
Thus, $\Qt_\beta=\CD_{a(\beta)}$ refines $\CD_{\tb}$.

(e)\  \ Fix $\beta\in\mathcal{I}$ and choose $\bn\in C_\beta\cap S_{L+d}$. Then  $\bn\preceq a(\be)$ and
$n_j$ and $a(\beta)_j$ belong to the same interval $I_{\beta_j}$. Thus,
\[
 0\le a(\beta)_j-n_j\le  \diam(I_{\be_j})=|I_{\be_j}|-1\leq h-1,
 \qquad 1\le j\le d.
\]
It follows that
\[
 |a(\beta)|_1\le|\bn|_1+d(h-1)=L+d+d(h-1).
\]
Using \eqref{7-2-b}, we obtain
\[
 |\Qt_\beta|=|\CD_{a(\be)}|
 =(2\ell_0)^d2^{|a(\beta)|_1}
 \le\ell_0^d2^{L+2d+d(h-1)}.
\]
Since $2^L\le2N$ and $h-1<(L+1)/q$, the last estimate implies
\[
 |\Qt_\beta|\le \ell_0^d 2^{2d+\f dq}  2^{L(1+\f dq)}\leq  C_d\ell_0^dN^{1+d/q}.
\]
Taking the maximum over $\beta\in\mathcal{I}$ proves \eqref{eq:macro-r}.

\end{proof}

\subsection{Transfer of the approximation}

Define
\[
 V:=\sum_{\beta\in\mathcal{I}}H_\beta.
\]

\begin{lemma}\label{prop:transfer}
For every $f\in\cT(N,d)$, there exists $F\in V$ such that
\[
 \norm{f-F\circ\pi}_{L^\infty(\T^d)}
 \le(L+1)^d\vartheta_0\norm f_\infty,
\]
where
\begin{equation}\label{eq:theta-star}
 \vartheta_0:=\frac{2\pi dA^d}{\ell_0}.
\end{equation}
\end{lemma}

\begin{proof}
Define
\[ \cA(f):=\Bl\{\bn\in \ZZ_+^d:\   \ \Psi_{\bn} f\   \ \text{is not identically zero}\Br\}.\]
 For each $\bn\in \cA(f)$, put $f_{\bn}=\Psi_{\bn} f$.
 By \eqref{eq:active-simplex},  $\cA(f)\subset M_{L+d}$. Hence,
 by  Lemma~\ref{lem:macro}(d), for every $\bn\in\cA(f)$ there is a retained index $\beta(\bn)\in\mathcal{I}$ such that
\[
 \bn\preceq a(\beta(\bn)),
\]
and hence $\Qt_{\beta(\bn)}$ refines $\CD_{\bn}$.

For every $Q\in\CD_{\bn}$, choose a point $\bx_{\bn,Q}\in Q$, and define
\[
 \wt g_{\bn}(\bx)
 :=\sum_{Q\in\CD_{\bn}}f_{\bn}(\bx_{\bn,Q})\1_Q(\bx).
\]
Lemma~\ref{lem:cellwise} gives
\[
 \norm{f_{\bn}-\wt g_{\bn}}_\infty
 \le\vartheta_0\norm f_\infty.
\]
The function $\wt g_{\bn}$ is constant on the cells of $\CD_{\bn}$, and therefore also on the cells of $\Qt_{\beta(\bn)}$. By \eqref{8-1}, there exists a unique  $g_{\bn}\in H_{\beta(\bn)}$ such that $\wt g_{\bn}=g_{\bn}\circ \pi$.

  Now  define
\[
  F:=\sum_{\bn\in\cA(f)}g_{\bn}.
\]
Then
\[ F\in \sum_{\bn\in\cA(f)}  H_{\beta(\bn)}\subset V.\]
By Lemma~\ref{lem:dlvp}(i) and $|\cA(f)|\le(L+1)^d$,
\begin{align*}
 \norm{f-F\circ\pi}_\infty
 &\le\sum_{\bn\in\cA(f)}\norm{f_{\bn}-\wt g_{\bn}}_\infty\le(L+1)^d\vartheta_0\norm f_\infty.
\end{align*}

\end{proof}

We shall also use the following standard probability amplification argument.

\begin{lemma}\label{lem:amplify}
Let $B>1$ and $0<\eta<\eta_0<1$. Then
\beq\label{8-5}
 m_*(N,d,B,\eta)
 \le\left\lceil\frac{\log(1/\eta)}{\log(1/\eta_0)}\right\rceil
 m_*(N,d,B,\eta_0).
\enq
\end{lemma}

\begin{proof}
Put
\[
 m_0:=m_*(N,d,B,\eta_0),
 \qquad
 k:=\left\lceil\frac{\log(1/\eta)}{\log(1/\eta_0)}\right\rceil.
\]
Let $\xi^1, \xi^2, \ldots$ be iid random points drawn uniformly from $\T^d$.
Divide the first $km_0$ sample points into the independent blocks
\[
 X^{(j)}:=\{\xi^{(j-1)m_0+1},\ldots,\xi^{jm_0}\},
 \qquad 1\le j\le k.
\]
For each $j$, let $E_j$ be the event that $X^{(j)}$ is not a $B$-norming set for $\cT(N,d)$. Then $\Prb(E_j)\le\eta_0$, and the events $E_1,\ldots,E_k$ are independent. Hence
\[
 \Prb(E_1\cap\cdots\cap E_k)\le\eta_0^k\le\eta.
\]
Let $X=\bigcup_{j=1}^k X^{(j)}$. If at least one of the blocks $X^{(j)}$ is a $B$-norming set, then their union  $X$ is also a $B$-norming set.  It follows that
 \eq{\PP& \Bl\{  \|f\|_\infty\leq B\|f\|_{X}, \   \ \forall f\in\cT(N,d)\Br\}\\
 &\ge \PP \Bl( \bigcup_{j=1}^k E_j^c\Br)=\PP \Bl[\big(\bigcap_{j=1}^k E_j\big)^c\Br]=1-\Prb(E_1\cap\cdots\cap E_k)\\
  &\ge 1-\eta_0^k\ge 1-\eta.
 }
 This means that  $X$ is a $B$-norming set for $\cT(N, d)$ with probability at least $1-\eta$, and hence proves \eqref{8-5}.

\end{proof}

\subsection{Proof of Theorem~\ref{thm:main}}\label{sec:proof}

First, we specify parameters. Fix $d\ge2$ and $\varepsilon\in(0,1)$. Set
\[
 q:=\left\lceil\frac{8d}{\varepsilon}\right\rceil,
 \qquad
 \theta:=\frac{\varepsilon}{8},
 \qquad
 K:=dq^{d-1},
 \qquad
 \ell:=\left\lceil\frac3\theta\right\rceil,
\]
and let
\[
 B_0:=2\ell M_K^\ell.
\]
All these quantities depend only on $d$ and $\varepsilon$.

Let $B^{(0)}_{d,\varepsilon}$ be the norming constant furnished by Corollary~\ref{cor:N2}. If $q>\sqrt{L+1}$, then $N<2^{q^2}$, and hence the extra factor $N$ in the estimate of Corollary~\ref{cor:N2} can be absorbed into a constant depending only on $d$ and $\varepsilon$. Since $1+\log(1/\eta)\le C\log(2/\eta)$, that corollary proves Theorem~\ref{thm:main} in this range after increasing $C(d,\varepsilon)$. We may therefore assume that
\[
 q\le\sqrt{L+1}.
\]
We first prove the result in this range with failure probability $\eta_0=1/2$.
Choose
\begin{equation}\label{eq:ell0-choice}
 \theta_*:=\frac1{4(B_0+1)(L+1)^d},
 \qquad
 \ell_0:=\left\lceil\frac{2\pi dA^d}{\theta_*}\right\rceil.
\end{equation}
Then $\vartheta_0\le\theta_*$, where $\vartheta_0$ is the constant in \eqref{eq:theta-star}.

Enumerate the retained indices as
\[
 \mathcal{I}=\{\beta_1,\ldots,\beta_{K'}\},
 \qquad \text{where}\ \  K'\le K.
\]
Define $\Pt_j:=\Qt_{\beta_j}$ for $1\leq j\leq K'$, and
$\Pt_j:=\Qt_{\beta_1}$ for $K'<j\le K$.
Repetitions are permitted in Theorem~\ref{thm:abstract}. Let $P_j$ be the conditional expectation associated with $\Pt_j$, and let $H_j$ be the space of all $\Pt_j$-constant functions on $\Omega$. By Lemma~\ref{lem:macro},
\[
 V=H_1+\cdots+H_K,
 \qquad
 P_iP_j=P_jP_i,
\]
and
\begin{equation}\label{eq:rbound}
 r:=\max_{1\le j\le K}|\Pt_j|
 \le C_d\ell_0^dN^{1+d/q}.
\end{equation}
 Thus we may  apply Theorem~\ref{thm:abstract} with exactly $K=dq^{d-1}$ partitions.

Let $\xi^1,\ldots,\xi^{m_0}$ be independent random points distributed according to $\mu_d$. Then $u^i=\pi(\xi^i)$, $i=1,\ldots, m_0$ are iid random elements uniformly distributed  on $\Omega$.  The partitions $\Pt_1,\dots,\Pt_K$ constructed above are
regular partitions of $\Omega$ with commuting conditional expectations. Hence, by Theorem~\ref{thm:abstract}, if
\begin{equation}\label{8-4}
 m_0\ge3C_K^2K^2r^{1+\theta}\log(4Kr),
\end{equation}
then, with probability at least $1/2$,
\begin{equation}\label{eq:Vnorming}
 \norm F_{\ell^\infty(\Omega)}
 \le B_0\max_{1\le i\le m_0}|F(u_i)|,
 \qquad \forall F\in V.
\end{equation}

Fix a realisation for which \eqref{eq:Vnorming} holds. Let $f\in\cT(N,d)$ and assume first that $\norm f_\infty=1$. By Lemma~\ref{prop:transfer}, there exists $F\in V$ such that
\[
 \norm{f-F\circ\pi}_\infty
 \le(L+1)^d\theta_*
 =\frac1{4(B_0+1)}=:E.
\]
Using \eqref{eq:Vnorming}, we obtain
\[
1=\norm f_\infty\le\norm{F\circ\pi}_{L^\infty(\T^d)}+E=\norm{F}_{\ell^\infty(\Omega)}+E
\le B_0\max_{1\leq i\leq m_0}\abs{F(u_i)}+E.
\]
However, for each $1\leq i\leq m_0$,
\[
\abs{F(u_i)}=\abs{F(\pi(\xi^i))}\le\abs{f(\xi^i)}+E\le\norm f_{X_{m_0}}+E .
\]
Combining these estimates yields
\[
\norm f_\infty=1\le B_0\norm f_{X_{m_0}}+(B_0+1)E=B_0\norm f_{X_{m_0}}+\tfrac14,
\]
which implies
\[
\norm f_\infty\le\tfrac43 B_0\norm f_{X_{m_0}}.
\]
By homogeneity, the same estimate holds for every $f\in\cT(N,d)$. Thus, with probability at least $1/2$,
\begin{equation}\label{eq:base-random-norming}
 \norm f_\infty
 \le B^{(1)}_{d,\varepsilon}\norm f_{X_{m_0}},
 \qquad f\in\cT(N,d),
\end{equation}
where
\[
 B^{(1)}_{d,\varepsilon}:=\f {4B_0}3=\f {8\ell M_K^\ell}3.
\]

It remains to estimate the right-hand side of \eqref{8-4}. By \eqref{eq:ell0-choice},
\[
 \ell_0
 \le1+8\pi dA^d(B_0+1)(L+1)^d
 \le C(d,\varepsilon)(L+1)^d.
\]
Hence \eqref{eq:rbound} gives
\begin{equation}\label{eq:rbound-final}
 r\le C(d,\varepsilon)(L+1)^{d^2}N^{1+d/q}.
\end{equation}
Since $d/q\le\varepsilon/8$ and $\theta=\varepsilon/8$,
\begin{equation}\label{eq:exponent}
 \left(1+\frac dq\right)(1+\theta)
 \le\left(1+\frac\varepsilon8\right)^2
 \le1+\frac\varepsilon3.
\end{equation}
It follows from \eqref{eq:rbound-final} that
\[
 r^{1+\theta}\log(4Kr)
 \le C(d,\varepsilon)
 N^{1+\varepsilon/3}\bigl(\log(2N)\bigr)^{2d^2+1}.
\]
Using the elementary estimate
\[
 \bigl(\log(2N)\bigr)^{2d^2+1}
 \le C(d,\varepsilon)N^{2\varepsilon/3},
 \qquad N\ge2,
\]
we conclude that \eqref{8-4} holds for some
\begin{equation}\label{eq:m0-bound}
 m_0\le C(d,\varepsilon)N^{1+\varepsilon}.
\end{equation}
This proves the required estimate when the failure probability is $1/2$.

We now let $\eta\in(0,1)$. If $\eta<1/2$, apply Lemma~\ref{lem:amplify} with $\eta_0=1/2$; if $\eta\ge1/2$, use the preceding estimate directly. In either case, \eqref{eq:m0-bound} gives
\[
 m_*(N,d,B^{(1)}_{d,\varepsilon},\eta)
 \le C(d,\varepsilon)N^{1+\varepsilon}\log\frac2\eta.
\]
Since adding sample points cannot destroy the norming property, the same conclusion holds whenever
\[
 m\ge C(d,\varepsilon)N^{1+\varepsilon}\log\frac2\eta.
\]
This proves the required sampling estimate in the range $q\le\sqrt{L+1}$.

Finally, we estimate the norming constant and combine the two ranges of $N$. Set
\[
 B_{d,\varepsilon}:=\max\{B^{(0)}_{d,\varepsilon},B^{(1)}_{d,\varepsilon}\}.
\]
From the proof of Corollary~\ref{cor:N2}, one may take
\[
 B^{(0)}_{d,\varepsilon}
 \le4\bigl(C_2(1+2\varepsilon^{-1})\bigr)^d.
\]
On the other hand, since
\[
 q\le\frac{9d}{\varepsilon},
 \qquad
 \ell\le\frac{25}{\varepsilon},
\]
we have
\[
 K=dq^{d-1}\le(9d)^d\varepsilon^{-d+1},
 \qquad
 K\ell\le25(9d)^d\varepsilon^{-d}.
\]
Using $M_K\le4^K$, we obtain
\begin{align*}
 B^{(1)}_{d,\varepsilon}
 &=\f 83\ell M_K^\ell\\
 &\le\frac{100}{\varepsilon}\,4^{K\ell}
 \le\exp\bigl(C_d\varepsilon^{-d}\bigr).
\end{align*}
The same upper bound, after increasing $C_d$, holds for $B^{(0)}_{d,\varepsilon}$. Hence
\[
 B_{d,\varepsilon}
 \le\exp\bigl(C_d\varepsilon^{-d}\bigr).
\]
This completes the proof of Theorem~\ref{thm:main}.

\end{document}